\documentclass[11pt, a4paper]{amsart}
\usepackage{amsthm, amssymb, dsfont, a4}
\usepackage{hyperref}
\usepackage[all]{xy}
\usepackage[active]{srcltx}
\usepackage[short,nodayofweek]{datetime}
\usepackage{xcolor}

\usepackage{tikz}
\usetikzlibrary{arrows}

\definecolor{my-blue}{cmyk}{1,0.6,0,0}

\definecolor{my-green}{cmyk}{0.8,0,1,0.5}

\hypersetup{
colorlinks=true, 
linkcolor=my-blue, 
citecolor=my-green,
urlcolor=black
}

\numberwithin{equation}{section}

\newcommand\FF{{\mathbb F}}

\newcommand\GG{{\mathbb G}}
\newcommand\NN{{\mathbb N}}
\newcommand\ZZ{{\mathbb Z}}
\newcommand\QQ{{\mathbb Q}}
\newcommand\RR{{\mathbb R}}

\newcommand{\Fq}{{\FF_{\!q}}}

\newcommand{\cL}{\mathcal{L}}
\newcommand{\cO}{\mathcal{O}}
\newcommand{\Ga}{\GG_a}
\newcommand{\cK}{\mathcal{L}}

\newcommand{\mot}{\mathsf{M}}  % notation for the t-motive
\newcommand{\dumot}{\mathfrak{M}}  % notation for the dual t-motive

\newcommand{\grp}{\mathrm{grp}}
\newcommand{\mothat}{\hat{\mot}^{\sigma}} % completion of \mot w.r.t v_\sigma
\newcommand{\stlatt}{\Lambda_{\rm st}}  % standard \cO-lattice generated by fixed basis of \mot

\newcommand{\one}{\mathds{1}}
\newcommand{\transp}[1]{#1^{\rm tr}}   % transpose

\newcommand{\bigmid}{\, \Big| \,}
\newcommand{\dk}{\check{\kappa}} % "dual kappa" for basis of \dumot.

\newcommand{\pr}{\mathrm{pr}}
\newcommand{\inj}{\mathrm{in}}
\newcommand{\dphi}{d\phi}

\newcommand{\id}{\mathrm{id}}

\newcommand{\ord}{\mathrm{ord}}
\newcommand{\rk}{\mathrm{rk}}

\newcommand{\ps}[1]{[\![#1]\!]}  % brackets for power series 
\newcommand{\ls}[1]{(\!(#1)\!)}% brackets for laurent series
\renewcommand{\sp}[1]{\{#1\}} % brackets for skew polynomial ring
\newcommand{\sps}[1]{\{\!\{#1\}\!\}} % brackets for skew power series ring
\newcommand{\sls}[1]{(\!\{#1\}\!)}% brackets for skew laurent series

\newcommand{\svect}[2]{\left( \begin{matrix} {#1}_{1}\\ \vdots \\ {#1}_{#2}\end{matrix}\right)}
\newcommand{\zvect}[2]{\left( \begin{matrix} {#1}_{1}& {#1}_{2} & \cdots & {#1}_{#2}\end{matrix}\right)}

\newcommand{\betr}[1]{\lvert #1\rvert}
\newcommand{\norm}[1]{\left\lVert #1\right\rVert}
\newcommand{\scalar}[1]{\begin{smatrix} 
#1 & 0 & \cdots & 0\\ 
0 & \ddots & \ddots & \vdots\\
\vdots & \ddots & \ddots & 0 \\
0 & \cdots & 0 & #1 \end{smatrix}
}

\newenvironment{smatrix}{\left(\begin{smallmatrix}}{\end{smallmatrix}\right)}

\DeclareMathOperator{\Hom}{Hom}
\DeclareMathOperator{\End}{End}

\DeclareMathOperator{\Mat}{Mat}
\DeclareMathOperator{\GL}{GL}
\DeclareMathOperator{\Lie}{Lie}

\theoremstyle{plain}
\newtheorem{thm}{Theorem}[section]
\newtheorem{cor}[thm]{Corollary}
\newtheorem{lem}[thm]{Lemma}
\newtheorem{prop}[thm]{Proposition}

\newtheorem{mainthm}{Theorem}

\theoremstyle{definition}
\newtheorem{defn}[thm]{Definition}
\newtheorem{exmp}[thm]{Example}
\newtheorem{rem}[thm]{Remark}
\begin{document}

\title[Abelian equals A-finite]{Abelian equals A-finite for Anderson A-modules}
\author{Andreas Maurischat}
\address{\rm {\bf Andreas Maurischat}, FH Aachen University of Applied Sciences, RWTH Aachen University, Germany}
\email{\sf andreas.maurischat@matha.rwth-aachen.de}

%\date{\today}
\newdate{date}{18}{02}{2024}
\date{\displaydate{date}}

% AMS-Classification
%\classification{
% 11G09: Drinfel'd modules; higher-dimensional motives, etc.
% 11J93: Transcendence theory of Drinfel'd and $t$-modules, 
% 16W60: Valuations, completions, formal power series and related constructions (associative rings and algebras)
% 
%}

% Keywords
%\keywords{abelian, t-module, t-motive, skew field, Newton polygon}

%-----------------------------------------------------------  
\begin{abstract}
Anderson introduced t-modules as higher dimensional analogs of Drinfeld modules. Attached to such a t-module, there are its t-motive and its dual t-motive. The t-module gets the attribute ``abelian'' when the t-motive is a finitely generated module, and the attribute ``t-finite'' when the dual t-motive is a finitely generated module. The main theorem of this article is the affirmative answer to the long standing question whether these two attributes are equivalent. The proof relies on an invariant of the t-module and a condition for that invariant which is necessary and sufficient for both being abelian and being t-finite.
We further show that this invariant also provides the information whether the t-module is pure or not.
Moreover, we conclude that also over general coefficient rings A, i.e.~for Anderson A-modules, the attributes of being abelian and being A-finite are equivalent. 
%We furthermore prove a structural theorem on non-abelian t-modules, and introduce the notions of pseudo-abelian and almost abelian t-modules.
%Those are more general than abelian t-modules but still share a lot of properties. The study of those more general t-modules is motivated by the fact that they naturally arise e.g. as extensions of Drinfeld modules by $\GG_a$ which are considered for de Rham cohomology. These extensions are even almost abelian.
%
%Maybe we show that pseudo-abelian A-modules are regular in the sense of Yu \cite{jy:ahdm}. (or don't we?)
\end{abstract}
%-----------------------------------------------------------  

\maketitle

\setcounter{tocdepth}{1}
\tableofcontents
%---------------------------------------------------------------------
% Beginn des eigentlichen Artikels
%---------------------------------------------------------------------

\section*{Introduction}

Anderson $t$-modules are higher dimensional analogs of Drinfeld modules, and they are considered as the function field analogs in positive characteristic of abelian varieties. Of great interest are the periods of Anderson $t$-modules, logarithmic vectors, (multiple) zeta values and much more, and similar conjectures about the transcendence and algebraic independence are stated as for classical periods, logarithms etc.

\medskip

Let $K$ be perfect field containing the finite field $\Fq$ of $q$ elements. Roughly speaking, an Anderson $t$-module over $K$ of dimension $d$ is a certain pair $(E,\phi)$ where $E$ is an algebraic group over $K$ which is isomorphic to $\GG_{a}^d$ -- the $d$-fold product of the additive group --, equipped with a compatible $\Fq$-action, and 
\[\phi:\Fq[t]\to \End_{\grp,\Fq}(E), a\mapsto \phi_a\]
is a certain $\Fq$-algebra homomorphism into $\Fq$-linear endomorphisms of $E$.\footnote{See Section \ref{sec:t-modules} for the precise definition.}

\medskip

There are two kinds of motives attached to Anderson $t$-modules.
The first one -- called \emph{$t$-motive} -- was defined by Anderson in his seminal paper \cite{ga:tm}. Among others, it was used ibid.~to provide a new criterion for uniformizability of $t$-modules.
The other kind -- called \emph{dual $t$-motive} -- was defined later in \cite{ga-wb-mp:darasgvpc}, and is the base for the ABP-criterion \cite[Thm.~3.1.1]{ga-wb-mp:darasgvpc} and the consequence \cite[Thm.~5.2.2]{mp:tdadmaicl} which opened a new way to study algebraic independence of periods and logarithms. This criterion and a generalization of Chang \cite[Thm.~1.2]{cc:nrvabpc} was then successfully applied in several other papers (e.g.~\cite{cc-jy:daraszvpc,cc-mp-dt-jy:aiagvczv,cc-mp-jy:fdeaizvpec,cc-mp:aipldm,ym:aicppcmv,ym:aicmvcp,am:ptmaip,am:aicph,cn:arhpldm}).
In \cite{am:ptmsv}, we showed that the criteria mentioned above can also be applied to $t$-motives, and that the $t$-motive can also be used to study algebraic independence of periods.

For using the $t$-motive or the dual $t$-motive in that way, the $t$-module has to be uniformizable in both cases. For using the $t$-motive, in addition the $t$-module needs to satisfy a property called \textit{abelian} -- which apparently was included in Anderson's original definition in \cite{ga:tm}.
For using the dual $t$-motive, however, the $t$-module needs to satisfy another property called \textit{$t$-finite}.

\medskip

The \emph{$t$-motive} attached to such a $t$-module consists of the $\Fq$-linear group scheme homomorphisms from $E$ into the additive group $\Ga$,
\[ \mot(E):=\Hom_{\grp,\Fq}(E,\Ga), \]
and the \emph{dual $t$-motive} is defined as
\[ \dumot(E):=\Hom_{\grp,\Fq}(\Ga,E). \]
They both carry a structure as $K[t]$-module obtained from the $K$-action on $\Ga$ and the $\Fq[t]$-action $\phi$ on $E$.
The $t$-module $E$ is said to be \emph{abelian} if its $t$-motive $\mot(E)$ is a finitely generated $K[t]$-module, and the $t$-module $E$ is said to be \emph{$t$-finite} if its dual $t$-motive $\dumot(E)$ is a finitely generated $K[t]$-module.

\medskip

Although these two notions are quite similar, and all abelian examples considered up to date were indeed $t$-finite and vice versa, it was an open problem whether these two properties really agree. Our first main theorem (Theorem~\ref{mainthm:equivalence}) answers this question to the affirmative.

%\begin{mainthm}
%For an Anderson $t$-module $(E,\phi)$, the following are equivalent.
%\begin{enumerate}
%\item $E$ is abelian,
%\item $E$ is $t$-finite.
%\end{enumerate}
%\end{mainthm}

We should also mention here that this theorem implies the positive answer to the same question for Anderson $A$-modules, i.e.~over more general coefficient rings $A/\Fq[t]$ (see Corollary~\ref{cor:main-theorem-A-modules}).

\medskip

The proof of this theorem stems from a condition on the endomorphism $\phi_t$ which is necessary and sufficient for both being abelian and being $t$-finite.
The proof involves techniques over non-commutative rings which are well known in the commutative setting, but were only partially developed in this setting. Hence, we have to develop these techniques here, too.

For describing the condition on the endomorphism $\phi_t$, we have to provide more notation.
%Let $\tau:K\to K, a\mapsto a^q$ be the $q$-th power Frobenius map which is an isomorphism as $K$ is assumed to be perfect, and $\sigma=\tau^{-1}$ its inverse. 
We denote by $K\{\tau\}$ the skew polynomial ring
\[  K\sp{\tau}=\left\{ \sum_{i=0}^n \alpha_i\tau^i \,\middle|\, n\geq 0, \alpha_i\in K\right\} \]
with multiplication uniquely given by additivity and the rule
\[ \tau \cdot \alpha = \alpha^q\cdot \tau, \]
for all $\alpha\in K$.
%, i.e., 
%\[ \left( \sum_{i=0}^n \alpha_i\tau^i\right) \cdot \left(\sum_{j=0}^m \beta_j\tau^j\right)
%= \sum_{k=0}^{n+m} \left( \sum_{i=0}^k \alpha_i\cdot (\beta_{k-i})^{q^i} \right) \cdot \tau^k.  \]
This ring equals the ring $\End_{\grp,\Fq}(\Ga)$ of $\Fq$-linear group endomorphisms of $\Ga$ by identifying $\tau$ with the $q$-th power Frobenius map, and $\alpha\in K$ with scalar multiplication by $\alpha$.
We further consider the skew Laurent series ring over $K$ in $\sigma=\tau^{-1}$,
\[ K\sls{\sigma} :=\left\{ \sum_{i=i_0}^\infty \alpha_i\sigma^i \,\middle|\, i_0\in \ZZ, \alpha_i\in K\right\}, \]
with $\sigma \cdot \alpha = \alpha^{1/q}\cdot \sigma$ for all $\alpha\in K$ (well-defined as $K$ is assumed to be perfect), as well as the subring of skew power series in $\sigma$,
\[ K\sps{\sigma} :=\left\{ \sum_{i=0}^\infty \alpha_i\sigma^i \,\middle|\, \alpha_i\in K\right\}. \]

The ring $K\sp{\tau}$ is naturally embedded into $K\sls{\sigma}$ via
\[ \sum_{i=0}^n \alpha_i\tau^i \mapsto \sum_{i=0}^n \alpha_i\sigma^{-i}. \]

After choosing a coordinate system for the $t$-module $E$, the endomorphism $\phi_t$ is represented by a matrix $D\in \Mat_{d\times d}(K\sp{\tau})$ with respect to this coordinate system, and we consider $D$ inside $\Mat_{d\times d}(K\sls{\sigma})$ via the natural embedding $K\sp{\tau}\hookrightarrow K\sls{\sigma}$.

As for matrices over commutative fields, one can compute invariant factors of $D$, i.e.~polynomials
$\lambda_1,\ldots, \lambda_d\in K\sls{\sigma}[t]$ with $\lambda_1|\lambda_2|\ldots|\lambda_d$ (i.e.~$\lambda_i$ left-dividing and right-dividing $\lambda_{i+1}$) obtained by diagonalizing the matrix $C:=t\cdot \one_d-D$ via row and column operations. 
As in the commutative case, we attach to such a polynomial $\lambda=\sum_{i=0}^n a_it^i\in K\sls{\sigma}[t]$ a Newton polygon. It is defined to be the lower convex hull of the set of points
$P_{i}=\left(i,\ord_{\sigma}(a_{i})\right)$, $i=0,\ldots, n$ where $\ord_{\sigma}(a_i)$ denotes the order of $a_i$ in $\sigma$.
Although, the invariant factors $\lambda_1,\ldots, \lambda_d$ are not unique in the non-commutative case, the orders in $\sigma$ of their coefficients are unique, and hence the Newton polygons are well-defined invariants.

The criterion for being abelian and $t$-finite then states:

\begin{mainthm} (Theorem~\ref{thm:main-theorem}) \label{mainthm:equivalence}
For a $t$-module $E$ the following are equivalent:
\begin{enumerate}
\item $E$ is abelian,
\item $E$ is $t$-finite,
\item the Newton polygon of the last invariant factor $\lambda_d$ constructed above has positive slopes only. 
\end{enumerate}
\end{mainthm}

The beauty of this criterion is not only its theoretical consequence that abelian $t$-modules are $t$-finite and vice versa, but that the criterion provides an algorithm to check the property for any given $t$-module.

\medskip

The Newton polygon of $\lambda_d$ even provides more information. Namely, we show

\begin{mainthm} (Theorem~\ref{thm:purity}) \label{mainthm:purity}
Assume that the $t$-module $E$ is abelian (and hence $t$-finite). The $t$-motive $\mot=\mot(E)$ of $E$ is pure if and only if
the Newton polygon of $\lambda_d$ consists of only one edge. In this case, the weight of $\mot$ is equal to the reciprocal of the slope of this edge. 
\end{mainthm}

By \cite[Thm.~5.29(b)]{uh-akj:pthshcff}, for an abelian and $t$-finite $t$-module $E$, its $t$-motive $\mot$ is pure if and only if its dual $t$-motive $\dumot$ is pure. In this case, the weights of $\mot$ and $\dumot$ just differ by the sign. Hence, our theorem implies the analogous statement for the dual $t$-motive (see Corollary~\ref{cor:purity-dual}).

As we will mainly deal with modules over the skew field $K\sls{\sigma}$, and the purity of $\mot$ is a condition on the $K\ls{\frac{1}{t}}$-vector space $K\ls{\frac{1}{t}}\otimes_{K[t]} \mot$, one main part in this task is to show that the scalar extensions $K\ls{\frac{1}{t}}\otimes_{K[t]} \mot$ and $K\sls{\sigma}\otimes_{K\{\tau\}} \mot$ are naturally isomorphic when $E$ is abelian (see Proposition \ref{prop:iso-of-completions}).

\medskip

The paper is structured as follows. Sections \ref{sec:notation} - \ref{sec:matrices} are on linear algebra over complete discretely valued skew fields. Only from Section \ref{sec:t-modules} on, we will deal with $t$-modules.  

We start in Section \ref{sec:notation} with basic notation on complete discretely valued skew fields. Section \ref{sec:newton-polygons} is devoted to Newton polygons and factorization of polynomials over these skew fields. Sections \ref{sec:modules} are on modules over complete discretely valued skew fields and lattices in these modules. Thereafter, we consider matrices over these skew fields in Section \ref{sec:matrices} where we proof an equivalence of several conditions for such matrices (see Theorem~\ref{thm:equvialent-matrix-conditions}). The equivalence of these conditions is a main step in the proof of Theorem \ref{mainthm:equivalence}.
After introducing the notion of $t$-modules, their $t$-motives and their dual $t$-motives in Section \ref{sec:t-modules}, we provide the proof of Theorem~\ref{mainthm:equivalence} in Section \ref{sec:criterion}, and of Theorem~\ref{mainthm:purity} in Section \ref{sec:purity}.
We deduce the equivalence of being abelian and being $A$-finite for general coefficient rings $A$ in Section \ref{sec:general-coefficients}.
Finally in Section \ref{sec:examples}, we illustrate our main theorems by some examples.

Our main sources for non-commutative ring theory are \cite{pmc:sftgdr}, \cite{nj:tr}, and \cite{oo:tncp}.

%Another part of the paper deals with structural properties of the motive and the dual motive of Anderson $t$-modules in general. In these investigations, we also define the notions of pseudo-abelian and almost abelian (resp.~pseudo $t$-finite and almost $t$-finite) $t$-modules which are more general the abelian (and $t$-finite), but still share a lot of properties. 
%Almost abelian $t$-modules naturally occur as extensions of Drinfeld modules by $\Ga$ which are considered for de Rham cohomology.
%In this way, we recover the quasi-periods of Drinfeld modules as periods of almost abelian $t$-modules, and using the connecting isomorphisms in \cite{am:ptmsv}, we can put the fact that they are obtained as values of certain functions into a bigger picture.

\subsection*{Acknowledgement}

We thank the referee for his careful reading of the manuscript and for pointing out two flaws in an earlier version.

\section{Basics in non-commutative algebra}\label{sec:notation}

%Notation
\subsection{Complete discretely valued skew fields}

%\begin{itemize}
%\item $a$ element in $\Fq[t]$,
%\item $\alpha,\beta$ elements in $K$,
%\item $f,g,h$ or $\lambda_1,\lambda_2$ etc. elements in $\cL[t]$ or $\cK[t]$ with coeff. $f_i,g_i,h_i$ or $a_i,b_i$,
%\item $x,y$ elements in $\cL$, $\cO$.
%\end{itemize}
%In Sect.~2-5: $\cL$ discrete valued skew field\\
%in Sect. 6: $\cL=K\sls{\sigma}$

%\subsection{Complete discretely valued skew fields} \ 

We use the following terminology:\\
A \emph{skew field} (or \emph{division ring}) is an associative ring with unit such that every non-zero element has a multiplicative inverse.

\begin{defn}
A \emph{discrete valuation} on a skew field $\cK$ is a map $v:\cK\to \ZZ\cup \{\infty\}$ such that
\begin{enumerate}
\item $v(x)=\infty \Longleftrightarrow x=0$,
\item for all $x,y\in \cK$: $v(x+y)\geq \min \{v(x),v(y)\}$,
\item for all $x,y\in \cK$: $v(xy)=v(x)+v(y)$.
\end{enumerate}
A skew field $\cK$ together with a discrete valuation $v$ on $\cK$ is called a \emph{discretely valued skew field}.

The subset $\mathcal{O}:=\{x\in \cK\mid v(x)\geq 0\}$ of a discretely valued skew field $(\cK,v)$ is called the \emph{valuation ring} of $\cK$. 

An element $\sigma\in \mathcal{O}$ with $v(\sigma)=1$ is called a \emph{uniformizer}.
\end{defn}

\begin{rem}
Of course, if $\cK$ is commutative, this is just the definition of a discretely valued field, and of its valuation ring.
As in the commutative case, we have several conclusions from the definition of a discretely valued skew field $(\cK,v)$ with valuation ring $\mathcal{O}$.
\begin{itemize}
\item $v(1)=0$, and $v(x^{-1})=-v(x)$ for all $x\in \cK\smallsetminus \{0\}$,
\item for all $x,y\in \cK$ such that $v(x)\neq v(y)$: $v(x+y)=\min\{v(x),v(y)\}$,
\item $\mathcal{O}$ is a subring of $\cK$ with unique maximal left (and right) ideal\\ $\mathfrak{m}=\{x\in \cK\mid v(x)> 0\}$,
\end{itemize}
We further remark that conjugate elements have the same valuation, i.e.
\begin{itemize}
\item for all $x,y\in \cK$, $y\neq 0$: $v(y^{-1}xy)=v(x)$,
\item the valuation ring $\mathcal{O}$ is even central in $\cK$, i.e. stable under conjugation,
\item the valuation ring $\mathcal{O}$ is a skew principal ideal domain (cf.~Subsection \ref{subsec:skew-pids}).
\end{itemize}
\end{rem}

\begin{defn}
Let $(\cK,v)$ be a discretely valued skew field. We say that a sequence $(s_n)_{n\geq 0}$ of elements in $\cK$ converges to $s\in \cK$, if
\[\lim_{n\to \infty} v(s-s_n)=\infty.\]
A discretely valued skew field $(\cK,v)$ is called \emph{complete}, if for every sequence $(x_n)_{n\geq 0}$ of elements in $K$ that tends to zero, there is an element $s\in \cK$ such that the sequence $\left(\sum_{i=0}^n x_i\right)_{n\geq 0}$ tends to $s$.
We will denote the limit $s$ by the infinite series $\sum_{i=0}^\infty x_i$.
\end{defn}

\begin{rem}
The definition of being \emph{complete} is just the usual one of metric space, i.e.~that all Cauchy sequences converge, if we equip $\cK$ with the absolute value
\[ \betr{x}=\rho^{v(x)} \text{ for } x\in \cK\smallsetminus \{0\},\quad \betr{0}=0, \]
for some $0<\rho<1$.
\end{rem}

\begin{exmp}\label{exmp:skew-Laurent-series-ring}
The standard example for a complete discretely valued skew field is the skew Laurent series ring $\cK=K\sls{\sigma}$ given in the introduction with valuation $v$ given by
\[ v\left( \sum_{i=i_0}^\infty a_i\sigma^i\right) = \inf\{ i\in \ZZ \mid a_i\neq 0\}. \]
Its valuation ring is the ring $\mathcal{O}=K\sps{\sigma}$ of non-commutative power series, and $\sigma\in \mathcal{O}$ is a uniformizer. 

The other example used in this paper is the Laurent series ring $K\ls{\frac{1}{t}}$ which is even a (commutative) field.
\end{exmp}

\subsection{Skew principal ideal domains}\label{subsec:skew-pids}

%As we need structural theorems for modules over non-commutative principal ideal domains, we introduce them here. As these theorems are valid for left modules and right modules, we omit the prefix.

\begin{defn}
A \emph{skew principal ideal domain} (\emph{skew PID}) is a unital associative ring without zero-divisors such that every left ideal and every right ideal is principal, i.e. generated by one element.
\end{defn}

\begin{exmp}\label{exmp:skew-pids}\ 
\begin{itemize}
\item Every skew field is a skew PID.
\item If $\cK$ is a discretely valued skew field, its valuation ring $\cO$ is a skew PID. Namely, all ideals $\ne (0)$ are given by $\cO\cdot \sigma^s=\sigma^s\cdot \cO$ for a fixed uniformizer $\sigma\in \cO$ and any $s\geq 0$. 
\item The ring $\cK[t]$ of polynomials over a skew field $\cK$ in a central indeterminate $t$ (i.e.~an indeterminate that commutes with all elements in $\cK$) is a skew PID, too.
\end{itemize}

The examples of skew PIDs that we will use from Section \ref{sec:t-modules} on, are the skew field of non-commutative Laurent series $\cL=K\sls{\sigma}$ given in the introduction, as well as its valuation ring -- 
the non-commutative power series ring $\mathcal{O}=K\sps{\sigma}$. Further, we employ its subring of non-commutative polynomials $K\sp{\sigma}$, as well as the polynomial ring $K\sls{\sigma}[t]$ over the skew Laurent series field $K\sls{\sigma}$ in a central indeterminate $t$.
\end{exmp}

\begin{thm}\label{thm:modules-over-pids}
Let $\mathcal{O}$ be a skew PID.
\begin{enumerate}
\item \label{nj:thm16} (see \cite[Ch.~3, Theorem 16]{nj:tr})\\
Any rectangular matrix with entries in $\mathcal{O}$ can be transformed by elementary row and column operations into a matrix in diagonal form, and each diagonal entry is a total divisor (i.e. left divisor and right divisor) of the subsequent diagonal entry.
\item \label{nj:thm17} (Elementary divisor theorem; see \cite[Ch.~3, Proof of Theorem 17]{nj:tr})\\ Let $M$ be a finitely generated free $\mathcal{O}$-module, and let $M'$ be a submodule of $M$. Then there are a basis $\{b_1,\ldots, b_n\}$ of $M$, elements $x_1,\ldots,x_n\in\mathcal{O}$, and $k\leq n$ such that $\{x_1b_1,\ldots, x_kb_k\}$ is a basis for $M'$
(respectively $\{b_1x_1,\ldots, b_kx_k\}$ for right modules).
\item \label{nj:thm18} (see \cite[Ch.~3, Theorem 18]{nj:tr})\\
Every finitely generated module over $\mathcal{O}$ is the direct sum of its torsion submodule and a free submodule.
\item \label{nj:thm19} (see \cite[Ch.~3, Theorem 19]{nj:tr})\\ Every finitely generated module over $\mathcal{O}$ is the direct sum of cyclic submodules.
\end{enumerate}
\end{thm}

From the first part of the previous theorem, and the characterization of ideals in valuation rings in Ex.~\ref{exmp:skew-pids}, we obtain the following corollary.

\begin{cor}\label{cor:diagonalising-over-valuation-ring}
Let $\mathcal{O}$ be the valuation ring of a discretely valued skew field $\cK$, and $\sigma\in \cO$ a uniformizer. Any rectangular matrix with entries in $\mathcal{O}$ can be transformed by elementary row and column operations into a matrix in diagonal form where the non-zero diagonal entries are of the form $\sigma^{\nu_1},\ldots, \sigma^{\nu_n}$ with $n\geq 0$, and $0\leq \nu_1\leq \ldots\leq \nu_n$.
\end{cor}

We close this section by defining what we mean by saying that a matrix $B$ has rank $r$ modulo $\sigma^s$.

\begin{defn}\label{def:rank-modulo-sigma-s}
Let $B$ be a rectangular matrix with entries in the valuation ring $\mathcal{O}$ of a discretely valued skew field $\cK$. Let $\sigma\in \cO$ be a uniformizer, and $s\geq 0$.
Let $\sigma^{\nu_1},\ldots, \sigma^{\nu_n}$ ($n\geq 0$, and $0\leq \nu_1\leq \ldots\leq \nu_n$) be the non-zero diagonal entries obtained by diagonalizing $B$ via elementary row and column operations as in the previous corollary.

We say that \emph{$B$ has rank $r$ modulo $\sigma^s$} if 
\[  \#\{ i\in \{1,\ldots, n\} \mid \nu_i<s \} = r.\]
%i.e., the number of $\nu_i$ that are smaller than $s$ equals $r$. 
\end{defn}

\section{Newton polygons}\label{sec:newton-polygons}

Throughout this section, let $(\cK,v)$ be a complete discretely valued skew field.
Further, let $\cK[t]$ be the polynomial ring over $\cK$ in a central indeterminate $t$, i.e. the indeterminate $t$ commutes with all elements in $\cK$.

In this section, we develop the theory of Newton polygons of polynomials in $\cK[t]$.
It runs parallel to the commutative setting, and the proofs are almost identical. As our coefficient ring is not commutative, there is a ``left-version'' and a ``right-version'' most of the times. However, one version can always be obtained from the other by applying it to the opposite ring of $\cK$. The opposite ring of $\cK$ is, by definition, the same additive group, but with multiplication $*$ given by $x*y:=y\cdot x$.

\begin{defn}
Let $f=\sum_{i=0}^n a_it^i\in \cK[t]$ be a polynomial. The Newton polygon of $f$ is defined to be the 
lower convex hull of the set of points
$P_{i}=\left(i,v(a_{i})\right)$, $i=0,\ldots, n$
ignoring the points with $a_{i}=0$.
We denote the Newton polygon of $f$ by $N_f$.

If $s$ is an edge of the Newton polygon from one vertex $(i,w_i)$ to another vertex $(j,w_j)$, we call the difference $|j-i|$ the \emph{length} of the edge, and the quotient $\frac{w_j-w_i}{j-i}$ the \emph{slope} of the edge. 
A vertex where two edges meet will be called a \emph{break point} of the Newton polygon.
\end{defn}

\begin{exmp}
For $f=a_0+a_1 t+a_2t^2+a_3t^3+t^5$ with $v(a_0)=3$, $v(a_1)=2$ and $v(a_2)=v(a_3)=1$ the Newton polygon $N_f$ consists of two edges. The first edge is of length $2$ with slope $\frac{v(a_2)-v(a_0)}{2}=-1$, and the second of length $3$ with slope $\frac{v(1)-v(a_2)}{3}=-\frac{1}{3}$ (see Figure \ref{fig:1}).

\noindent
\begin{figure}[ht] 
\begin{tikzpicture}[
    scale=0.8,
    axis/.style={thin, ->, >=stealth'},
    every node/.style={color=black},
    ]
	%\scriptsize 
    \tiny

	% Axis 
    \draw[axis] (-1,0)  -- (6,0) ;
    \draw[axis] (0,-1) -- (0,5) ;
    \draw[semithick] (1,-0.1) -- (1,0.1); 
    \draw[semithick] (2,-0.1) -- (2,0.1); 
    \draw[semithick] (3,-0.1) -- (3,0.1); 
    \draw[semithick] (4,-0.1) -- (4,0.1); 

    \draw[semithick] (-0.1,1) -- (0.1,1); 
    \draw[semithick] (-0.1,2) -- (0.1,2); 
    \draw[semithick] (-0.1,3) -- (0.1,3); 
    \draw[semithick] (-0.1,4) -- (0.1,4); 

	% Points
    \draw [fill] (0,3) circle [radius=.081] node [below left] (0,3) {$(0,v(a_0))$};
    \draw [fill] (1,2) circle [radius=.081] node [above right] (1,2) {$(1,v(a_1))$};
    \draw [fill] (2,1) circle [radius=.081] node [below left] (2,1) {$(2,v(a_2))$};
    \draw [fill] (3,1) circle [radius=.081] node [above right] (3,1) {$(3,v(a_3))$};
    \draw [fill] (5,0) circle [radius=.081] node [below] (5,0) {$(5,0)$};

    % Lines
    \draw[thick] (0,3) -- (2,1) -- (5,0); 
\end{tikzpicture} 
\caption{Newton polygon for $\sigma^3+\sigma^2\cdot t+\sigma\cdot t^2+\sigma\cdot t^3+t^4$.}\label{fig:1}
\end{figure}
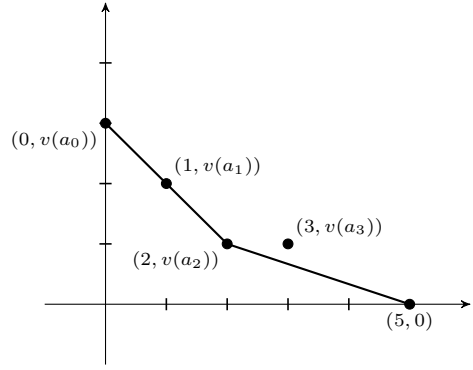
\end{exmp}

\begin{figure}[ht] 
\begin{tikzpicture}[
    scale=0.8,
    axis/.style={thin, ->, >=stealth'},
    every node/.style={color=black},
    ]
	%\scriptsize 
    \tiny

	% Axis 
    \draw[axis] (-1,0)  -- (5,0) ;
    \draw[axis] (0,-1) -- (0,5) ;
    \draw[semithick] (1,-0.1) -- (1,0.1); 
    \draw[semithick] (2,-0.1) -- (2,0.1); 
    \draw[semithick] (3,-0.1) -- (3,0.1); 
    \draw[semithick] (4,-0.1) -- (4,0.1); 

    \draw[semithick] (-0.1,1) -- (0.1,1); 
    \draw[semithick] (-0.1,2) -- (0.1,2); 
    \draw[semithick] (-0.1,3) -- (0.1,3); 
    \draw[semithick] (-0.1,4) -- (0.1,4); 

	% g(t)
    \draw [fill,color=blue] (0,3) circle [radius=.081] ;
    \draw [fill,color=blue] (1,1) circle [radius=.081];
    \draw [fill,color=blue] (4,0) circle [radius=.081] ;
    \draw [thick,color=blue] (0,3) -- (1,1) -- (4,0); 
	% f(t)
    \draw [fill,color=black] (0,1) circle [radius=.081] ;
    \draw [fill,color=black] (1,0) circle [radius=.081] ;
    \draw [thick,color=black] (0,1) -- (1,0); 
	% h(t)=g*f
    \draw [fill,color=green] (0,4) circle [radius=.081] ;
    \draw [fill,color=green] (1,2) circle [radius=.081];
    \draw [fill,color=green] (2,1) circle [radius=.081];
    \draw [fill,color=green] (4,1) circle [radius=.081] ;
    \draw [fill,color=green] (5,0) circle [radius=.081] ;
    \draw [thick,color=green] (0,4) -- (1,2) -- (2,1) -- (5,0); 
    
\end{tikzpicture} 
\caption{Newton polygons for $g(t)=\sigma^3+\theta\sigma\cdot t+t^4$ (blue), $f(t)=\sigma+t$ (black), and 
$h=g\cdot f=\sigma^4+(\theta\sigma^2+\sigma^3)t+\theta\sigma t^2+\sigma t^4+t^5$ (green).}\label{fig:2}
\end{figure}
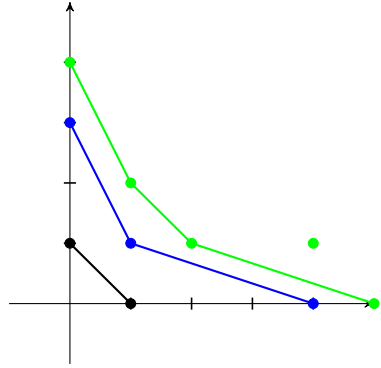

\begin{prop}
Let $f,g\in \cK[t]$ be polynomials. Then the Newton polygon $N_{gf}$ of $g\cdot f$ starts at the vector sum of the starting points of $N_f$ and $N_g$, and ends at the vector sum of the end points of $N_f$ and $N_g$. The edges of the Newton polygon $N_{gf}$ are exactly the edges of both $N_f$ and $N_g$ concatenated in non-decreasing order of their slopes.

The same holds for the product $f\cdot g$.
\end{prop}

\begin{proof}
The proof is the same as in the commutative case. Also compare the example in Figure \ref{fig:2}.
\end{proof}

As our skew field $\cK$ is complete with respect to the valuation, we will also obtain a reverse result in Prop.~\ref{prop:factoring-by-slopes}. For its proof, however, we need some results on division with remainder in our polynomial ring $\cK[t]$.

\begin{defn}\label{def:v_c}
For $c\in \RR$, we define a valuation $v_c$ on $\cK[t]$ by
\[  v_c\left(\sum_{i=0}^n h_it^i\right) = \min\left\{ v(h_i)+i\cdot c \mid i\in \{0,\ldots, n\} \right\}\in \RR\cup \{\infty\}.\]
It is not hard to verify that this is indeed a valuation, i.e.~it satisfies
\begin{itemize}
\item $v_c(f)=\infty \Longleftrightarrow f=0$,
\item $v_c(f\cdot g)=v_c(f)+v_c(g)$,
\item $v_c(f+g)\geq \min\{ v_c(f),v_c(g)\}$.
\end{itemize}
\end{defn}

\begin{lem} (Right-division with remainder)\\
Let $c\in \RR$ be arbitrary. Let $h,f\in \cK[t]$ with $\deg_t(h)\geq \deg_t(f)$.
\begin{enumerate}
\item There exist unique $q,r\in \cK[t]$ such that 
\[ h=q\cdot f+r, \]
and either $r=0$ or $\deg_t(r)<\deg_t(f)$.
\item \label{lem:item:2} Assume further that $v_c(f)=v(f_d)+d\cdot c$, where $d=\deg_t(f)$ and $f_d$ is the leading coefficient of $f$. Then
\[ v_c(q)\geq v_c(h)-v_c(f)\quad \text{and}\quad v_c(r)\geq v_c(h). \]
\end{enumerate}
\end{lem}

\begin{proof}
The existence of $q$ and $r$ is shown in \cite[Sect.~2]{oo:tncp}, and uniqueness is obtained in the same way as in the commutative case. So it remains to show the bounds on the valuations of $q$ and $r$ in \eqref{lem:item:2}.\\
Write $h=\sum_{k=0}^n h_kt^k$ where $n=\deg_t(h)$, $f=\sum_{j=0}^d f_jt^j$, and
$q=\sum_{i=0}^{n-d} q_it^i$. As $\deg_t(r)<d$, we obtain for all $e>d$:
\begin{equation}\label{eq:h_e}
 h_e=\sum_{i+j=e} q_if_j=q_{e-d}\cdot f_d+\sum_{j=0}^{d-1} q_{e-j}f_j. 
\end{equation}
We show $v(q_i)+i\cdot c\geq v_c(h)-v_c(f)$ for all $i=n-d,\ldots, 0$ by backwards induction which then implies $v_c(q)=\min \{v(q_i)+i\cdot c \mid 0\leq i\leq n-d\}\geq v_c(h)-v_c(f)$.\\
For $i=n-d$, set $e=n$ in equation \eqref{eq:h_e}. It reduces to $h_n= q_{n-d}\cdot f_d$, and therefore
\[ v(q_{n-d})=v(h_n)-v(f_d)=v(h_n)-v_c(f)-d\cdot c\geq v_c(h)+n\cdot c-v_c(f)-d\cdot c=
v_c(h)-v_c(f)-(n-d)\cdot c.\]
Now, fix $i<n-d$, and assume by induction hypothesis that we have shown the inequality for all larger indices. Then setting $e=i+d$ in equation \eqref{eq:h_e}, we obtain
\begin{eqnarray*}
v(q_{i})&=& v\left(h_{i+d}-\sum_{j=0}^{d-1} q_{i+d-j}f_j\right)-v(f_d) \\
&\geq & \min \{ v(h_{i+d}), v(q_{i+d-j})+v(f_j) \mid j=0,\ldots, d-1\} - v(f_d) \\
%&\geq & \min \{ v_c(h)-(i+d)\cdot c, v_c(h)-v_c(f)-(i+d-j)\cdot c+v_c(f)-j\cdot c \mid j=0,\ldots, d-1\} - v(f_d) \\
&\geq & \min \{ v_c(h)-(i+d)\cdot c, v_c(h)-(i+d-j)\cdot c-j\cdot c \mid j=0,\ldots, d-1\} - v(f_d) \\
&=&  v_c(h)-(i+d)\cdot c - v_c(f)+d\cdot c = v_c(h)-v_c(f)-i\cdot c.
\end{eqnarray*}
The bound for the valuation of $r$ is directly computed as
\[ v_c(r)=v_c\left(h-q\cdot f\right)\geq \min\{ v_c(h), v_c(q)+v_c(f)\}\geq v_c(h),\]
using the bound for $v_c(q)$ just obtained.
\end{proof}

By the same arguments with sides swapped, we obtain the result on left-division with remainder.

\begin{lem}\label{lem:left-division-with-remainder} (Left-Division with remainder)\\
Let $c\in \RR$ be arbitrary. Let $h,f\in \cK[t]$ with $\deg_t(h)\geq \deg_t(f)$.
\begin{enumerate}
\item There exist unique $q,r\in \cK[t]$ such that 
\[ h=f\cdot q+r, \]
and either $r=0$ or $\deg_t(r)<\deg_t(f)$.
\item \label{lem:item:2l} Assume further that $v_c(f)=v(f_d)+d\cdot c$ where $d=\deg_t(f)$ and $f_d$ is the leading coefficient of $f$. Then
\[ v_c(q)\geq v_c(h)-v_c(f)\quad \text{and}\quad v_c(r)\geq v_c(h). \]
\end{enumerate}
\end{lem}

\begin{prop}\label{prop:factoring-by-slopes}
Let $h\in \cK[t]$, and assume that its Newton polygon has at least two edges, and that $P=(d,e)$ is the first break point of the Newton polygon of $h$ (i.e. the break point with the least $x$-coordinate of all break points). Then the following hold: 
\begin{enumerate}
\item There are polynomials $f,g\in \cK[t]$ such that $\deg_t(f)=d$, the Newton polygon of $f$ consists of the first edge of $N_h$, and $h=f\cdot g$.
\item There are polynomials $f,g\in \cK[t]$ such that $\deg_t(f)=d$, the Newton polygon of $f$ consists of the first edge of $N_h$, and $h=g\cdot f$.
\end{enumerate}
\end{prop}

\begin{rem}\label{rem:right-factor-for-each-edge}
By repeating the previous process for the factor $g$, we eventually get a factorization of $h$ as $h=f_1\cdot f_2\cdots f_k$ where the Newton polygon of each $f_j$ has exactly one edge, each corresponding to one edge of $N_h$, i.e.~having the same length and same slope as that edge of $N_h$.
Furthermore, for each permutation of the $k$ edges of $h$, we get such a factorization. 

In particular, for each edge of $N_h$ there is a right factor $\mu$ of $h$ whose Newton polygon has exactly one edge, and this edge has the same length and same slope as that edge of $N_h$.
\end{rem}

\begin{proof}[Proof of Prop.~\ref{prop:factoring-by-slopes}]
We only prove the first part, as the second is obtained in the same way after swapping sides in products.
Let $s$ be the slope of the first edge, set $c:=-s$, and consider the valuation $v_c$ given in Definition \ref{def:v_c}.

Let $h=\sum_{j=0}^n h_jt^j$, and define $f^{(0)}=\sum_{j=0}^d h_jt^j\in \cK[t]$ as well as $g^{(0)}=1\in \cK[t]$.
The choice of $c$ implies that $v_c(h)=v(h_0)=v(h_d)+d\cdot c$, that $v(h_j)+j\cdot c\geq v_c(h)$ for all $0<j<d$, and that $v(h_j)+j\cdot c>v_c(h)$ for all $j>d$.
Hence, $v_c(f^{(0)})=v_c(h)$, and
\[v_c\left(h-f^{(0)}\cdot g^{(0)}\right)=v_c\left(\sum_{j=d+1}^n h_jt^j\right) =v_c(h)+\alpha \]
for some $\alpha>0$.

Inductively, we are going to define sequences $(f^{(i)})_{i\geq 0}$ and $(g^{(i)})_{i\geq 0}$ of polynomials in $\cK[t]$ such that
\begin{enumerate}
\item[a)] $\deg_t(f^{(i)})=d$ with leading coefficient $h_d$, $v_c(f^{(i)})=v_c(h)$, and for all $i\geq 1$: $v_c(f^{(i)}-f^{(i-1)})\geq v_c(h)+\alpha\cdot i$.
\item[b)] $\deg_t(g^{(i)})\leq n-d$, $v_c(1-g^{(i)})\geq \alpha$ and for all $i\geq 1$: $v_c(g^{(i)}-g^{(i-1)})\geq \alpha\cdot i$.
\item[c)] $v_c\left(h-f^{(i)}\cdot g^{(i)}\right)\geq  v_c(h)+\alpha\cdot (i+1)$.
\end{enumerate}
As the skew field $\cK$ is complete, this implies that the sequences $(f^{(i)})_{i\geq 0}$ and $(g^{(i)})_{i\geq 0}$ converge in $\cK[t]$
to some polynomials $f$ and $g$. By the conditions on the sequences, the polynomials $f$ and $g$ satisfy the conditions $\deg_t(f)=d$, $\deg_t(g)\leq n-d$, and $h=f\cdot g$.
Finally, condition a) implies
$v_c(f-f^{(0)})>v_c(f^{(0)})$, hence the Newton polygon of $f$ is the same as the one of $f^{(0)}$, namely the first edge of $N_h$.

By choice of $f^{(0)}$ and $g^{(0)}$, the conditions a) - c) are fulfilled for $i=0$. Now take $i>0$, and assume that
 $f^{(i-1)}$ and $g^{(i-1)}$ are already constructed. Then the hypothesis of Lemma \ref{lem:left-division-with-remainder}\eqref{lem:item:2l} (left division with remainder) are fulfilled, and we obtain $q,r\in \cK[t]$ such that
 \[  h-f^{(i-1)}\cdot g^{(i-1)} = f^{(i-1)}\cdot q+r, \]
 with the given bounds on the valuations of $q$ and $r$. We set
\[  f^{(i)}:= f^{(i-1)}+r,\qquad g^{(i)}:= g^{(i-1)}+q,\]
and will verify that these polynomials satisfy the conditions a) - c).

As $\deg_t(r)<\deg_t(f^{(i-1)})=d$, we have $\deg_t(f^{(i)})=\deg_t(f^{(i-1)}+r)=d$, and the leading coefficient of $f^{(i)}$ is again $h_d$. Further, by the bounds on the coefficients in the left division with remainder
\[ v_c(f^{(i)}-f^{(i-1)})=v_c(r)\geq v_c\left( h-f^{(i-1)}\cdot g^{(i-1)}\right)\geq v_c(h)+\alpha \cdot i. \]
The latter also implies $v_c(f^{(i)})=\min \{ v_c(f^{(i-1)}), v_c(f^{(i)}-f^{(i-1)})\}=v_c(h)$.
This proves condition a). As $\deg_t(q)= n-d$, and $\deg_t(g^{(i-1)})\leq n-d$, also $\deg_t(g^{(i)})=\deg_t(g^{(i-1)}+q)\leq n-d$, and
\[ v_c(g^{(i)}-g^{(i-1)})=v_c(q)\geq v_c\left( h-f^{(i-1)}\cdot g^{(i-1)}\right)-v_c(f^{(i-1)})\geq  v_c(h)+\alpha \cdot i-v_c(f^{(i-1)})
=\alpha \cdot i.\]
Therefore, also
\[ v_c(1-g^{(i)})\geq \min\{ v_c(1-g^{(i-1)}), v_c(g^{(i-1)}-g^{(i)})\}\geq \alpha.\]
Finally,
\[  h-f^{(i)}\cdot g^{(i)}=h-f^{(i-1)}\cdot g^{(i)}-r\cdot g^{(i)}=r-r\cdot g^{(i)}=r\cdot (1-g^{(i)}),\]
and hence
\[v_c\left(  h-f^{(i)}\cdot g^{(i)}\right)=v_c(r)+v_c(1-g^{(i)})\geq v_c(h)+\alpha \cdot i+\alpha=v_c(h)+\alpha \cdot (i+1).\]
\end{proof}

Using the previous proposition, we have the following non-commutative variants of the Chinese Remainder Theorem
\begin{thm}\label{thm:chinese-remainder-theorem}
\begin{enumerate}
\item Let $h\in \cK[t]$ be a polynomial whose Newton polygon consists of exactly $k$ edges, and let $\mu_1,\ldots, \mu_k$ be right factors of $h$ each corresponding to a different edge, then there is a canonical isomorphism of left $\cK[t]$-modules
\[ \cK[t]/h\cK[t] \cong \cK[t]/\mu_1\cK[t]\oplus \ldots \oplus\cK[t]/\mu_k\cK[t]. \]
\item Let $h\in \cK[t]$ be a polynomial whose Newton polygon consists of exactly $k$ edges, and let $\nu_1,\ldots, \nu_k$ be left factors of $h$ each corresponding to a different edge, then there is a canonical isomorphism of right $\cK[t]$-modules
\[ \cK[t]h\backslash \cK[t] \cong \cK[t]\nu_1\backslash \cK[t]\oplus \ldots \cK[t]\nu_k\backslash \oplus\cK[t]. \]
\end{enumerate}
\end{thm}

This follows inductively from the following lemma.

\begin{lem}
Let $h\in \cK[t]$ be a polynomial with the first break point of the Newton polygon $N_h$ being at $P=(d,e)$. Let
$h=\nu\cdot \lambda_1$ and $h=\lambda_2\cdot \mu$ be factorizations of $h$ from Proposition \ref{prop:factoring-by-slopes} where the Newton polygons of 
$\nu$ and $\mu$ consist of the first edge of $h$. Then the canonical homomorphism of left $\cK[t]$-modules
\[ \cK[t]/h\cK[t] \cong \cK[t]/\lambda_1\cK[t]\oplus \cK[t]/\mu\cK[t], \]
is an isomorphism,
as well as the canonical homomorphism of right $\cK[t]$-modules
\[ \cK[t]h\backslash \cK[t] \cong \cK[t]\nu\backslash \cK[t]\oplus \cK[t]\lambda_2\backslash \cK[t] \]
is an isomorphism.
\end{lem}

\begin{proof}
As no edge of $\lambda_1$ has the same slope as the edge of $\mu$, the right greatest common divisor of $\lambda_1$ and $\mu$ is $1$.
Hence by \cite[Thm.~4]{oo:tncp}, there exist $r,s\in \cK[t]$ such that $r \lambda_1+s \mu=1$. 
An inverse to the canonical homomorphism of left modules is therefore given by $(f,g)\mapsto fs\mu+gr \lambda_1$.

In the same way, one obtains the isomorphism for the right modules.
\end{proof}

\section{Modules over complete discretely valued skew fields}\label{sec:modules}

As before, let $(\cL,v)$ be a complete discretely valued skew field with valuation ring $\cO$.
Throughout this section, $M$ will denote a finite dimensional left-$\cL$-vector space.

We stick here to left modules. However, all the statements and theorems in this section about left modules will also be true for right modules. This is clear as a right module is a left module for the opposite ring.

\begin{defn}
An \textit{$\cO$-lattice} $\Lambda$ in $M$ is a finitely generated $\cO$-submodule $\Lambda$ of $M$ containing an $\cL$-basis of $M$.

\end{defn}

We collect some facts on $\cO$-lattices which are well-known in the commutative setting, but are also valid in the non-commutative setting.

\begin{rem}\label{rem:lattices} \ 
\begin{enumerate}
\item If $\sigma\in \cO$ is a uniformizer, and $\Lambda$ an $\cO$-lattice in $M$, then
$\bigcup\limits_{n\geq 0} \sigma^{-n}\Lambda = M.$
\item As $\cO$ is a skew principal ideal domain, any $\cO$-lattice $\Lambda$ in $M$ is a free $\cO$-module generated by a suitable $\cL$-basis of $M$ (cf.~Thm.~\ref{thm:modules-over-pids}\eqref{nj:thm18}). 
\item Any $\cO$-submodule $\Lambda'$ of an $\cO$-lattice $\Lambda$ of the same rank is an $\cO$-lattice in $M$. (It is finitely generated free containing an $\cL$-basis of $M$ by~Thm.~\ref{thm:modules-over-pids}\eqref{nj:thm17}.)
\item For two  $\cO$-lattices $\Lambda$ and $\Lambda'$ in $M$, their intersection $\Lambda\cap \Lambda'$ as well as their sum $\Lambda+\Lambda'\subseteq M$ are  $\cO$-lattices, too.
\item If $\Lambda$ is an $\cO$-lattice of $M$, and $\Lambda'$ an $\cO$-submodule of $M$. Then $\Lambda'$ is an $\cO$-lattice in $M$, if and only if there are integers $l_1,l_2\in \ZZ$ such that
$\sigma^{l_1}\Lambda\subseteq \Lambda'\subseteq \sigma^{l_2}\Lambda$.
\end{enumerate}
\end{rem}

\begin{lem}\label{lem:O-lattice-description}
An $\cO$-submodule $\Lambda$ of $M$ is an $\cO$-lattice if and only if $\Lambda$ contains an $\cL$-basis of $M$, and for every $m\in M, m\ne 0$, there is $x\in \cL$ such that $x\cdot m\not\in \Lambda$. 
\end{lem}

\begin{proof}
For showing the first implication, let $\Lambda$ be an $\cO$-lattice, and $b_1,\ldots, b_d$ be a basis of $\Lambda$. Let $m\in M$ be arbitrary, then there are $x_1,\ldots, x_d\in \cL$ such that $m=\sum_{i=1}^d x_ib_i$. Choose $x\in \cL$ such that $v(x)< -v(x_1)$. Then $xm=\sum_{i=1}^d (xx_i)b_i\not\in \Lambda$ as $v(xx_1)<0$, i.e. $xx_1\not\in \cO$.

For the other implication, let $\Lambda$ be an $\cO$-submodule of $M$ containing an $\cL$-basis of $M$, and satisfying the condition that for every $m\in M$, there is $x\in \cL$ such that $x\cdot m\not\in \Lambda$. 
Choose an $\cO$-lattice $\tilde{\Lambda}$ in  $M$ and consider the $\cO$-lattices $\Lambda_i$ of $M$ given by
\[ \Lambda_i:=\sigma^i\Lambda \cap \tilde{\Lambda} \]
for all $i\geq 0$, where $\sigma\in \cO$ is a uniformizer, as well as
\[ \bar{\Lambda}_i:=(\Lambda_i+\sigma\tilde{\Lambda})/(\sigma\tilde{\Lambda}) \subseteq \tilde{\Lambda}/(\sigma\tilde{\Lambda}).  \]
The latter are vector spaces over the skew field $\cO/\sigma\cO$ of dimension at most $d$.
As by definition, the sequence $(\Lambda_i)_{i\geq 0}$ is a desending chain of $\cO$-modules, the sequence
$(\bar{\Lambda}_i)_{i\geq 0}$ is a descending chain of $(\cO/\sigma\cO)$-vector spaces. Hence, there is some $n_0\geq 0$ such that
$\bar{\Lambda}_n=\bar{\Lambda}_{n_0}$ for all $n\geq n_0$.

Assume that $\bar{\Lambda}_{n_0}\neq 0$, then there is $m\in \Lambda_{n_0} \setminus \sigma\tilde{\Lambda}$, i.e.~$0\neq m\in \Lambda_{n}=\sigma^n\Lambda \cap \tilde{\Lambda}$ for all $n\in \NN$.
However, this implies $\sigma^{-n}m\in \Lambda$ for all $n\geq n_0$. As $\Lambda$ is an $\cO$-module, we get $\sigma^{-n}m\in \Lambda$ for all $n\in \ZZ$, contradicting our assumption.

Therefore $\bar{\Lambda}_{n}=0$ for all $n\geq n_0$, i.e.~$\Lambda_{n}+\sigma\tilde{\Lambda}=\sigma\tilde{\Lambda}$, or in other words $\Lambda_{n}\subseteq \sigma\tilde{\Lambda}$. Hence
\[ \Lambda\cap \sigma^{-n}\tilde{\Lambda}\subseteq \sigma^{1-n}\tilde{\Lambda} \]
for all $n\geq n_0$. Inductively, one obtains
\[ \Lambda\cap \sigma^{-n}\tilde{\Lambda} = \Lambda\cap \sigma^{-n_0}\tilde{\Lambda}\]
for all $n\geq n_0$, and by taking the union over all $n$:
\[ \Lambda =  \Lambda\cap \sigma^{-n_0}\tilde{\Lambda},\]
i.e.~$\Lambda\subseteq \sigma^{-n_0}\tilde{\Lambda}$ which implies that $\Lambda$ is finitely generated.
\end{proof}

\begin{prop}\label{prop:integral-submodule}
Let $\Lambda$ be an $\cO$-lattice inside $M$.
Further, let $M'$ be an $\cL$-submodule of $M$, and define
$\Lambda':=\Lambda\cap M'$.
Then any $\cO$-basis of $\Lambda'$ can be extended to an $\cO$-basis of $\Lambda$.
\end{prop}

\begin{proof}
As $\cO$ is a skew PID, we just have to show that $\Lambda/\Lambda'$ is torsionfree, hence free. Then the join of a basis of $\Lambda'$ and representatives of a basis of $\Lambda/\Lambda'$ form a basis of $\Lambda$.

Assume to the contrary, that $\Lambda/\Lambda'$ is not torsionfree. Hence, there exist $m\in \Lambda - \Lambda'$, and $x\in \cO$ such that $x\cdot m\in \Lambda'$. As 
$\Lambda'=\Lambda\cap M'$, this implies $x\cdot m\in M'$, and hence $m\in M'$, leading to the contradiction $m\in \Lambda\cap M'=\Lambda'$. 
\end{proof}

\begin{rem}
Be aware that in general we can not extend this result to a decomposition $M=M'\oplus M''$ of submodules,
i.e. we can not find bases of the $\cO$-sublattices $\Lambda'$ and $\Lambda''$ which join to a basis of $\Lambda$. 
\end{rem}

We now turn our attention to $\cL[t]$-modules which are finite dimensional as $\cL$-vector spaces.
We start by stating the invariant factors theorem which follows directly from the proof of Thm.~\ref{thm:modules-over-pids}\eqref{nj:thm19} in \cite{nj:tr}, and is also explained in \cite[p.~380]{pmc:sftgdr}.

\begin{thm}\label{thm:elementary-divisor-theorem} (Invariant Factors Theorem, see \cite[p.~380]{pmc:sftgdr})
Given a left $\cL[t]$-module $M$ which is finitely generated as $\cL$-vector space of dimension $d$,
there exist monic polynomials $\lambda_1,\ldots,\lambda_d\in \cL[t]\smallsetminus\{0\}$ such that $\lambda_{i+1}$ is left-divisible and right-divisible by $\lambda_i$ for all $i=1,\dots,d-1$, and
\[  M \cong \cL[t]/\lambda_1\cL[t]\oplus \ldots \oplus\cL[t]/\lambda_d\cL[t] \]
as $\cL[t]$-modules.
\end{thm}

\begin{rem}\label{rem:invariant-factor-computation}
Contrary to the commutative case, the monic polynomials $\lambda_i$ are not unique, but only unique up to \emph{similarity} which is equivalent to the factors $\cL[t]/\lambda_i\cL[t]$ being unique up to isomorphism (cf.~\cite[Ch.~3, Theorem 31]{nj:tr}). However, the
way to obtain them is the same as in the commutative case:\\
Choose an $\cL$-basis $(e_1,\ldots, e_d)$ of $M$, represent multiplication by $t$ by a matrix $D$ with respect to this basis, and
consider the matrix $C:=t\cdot \one -D\in \Mat_{d\times d}(\cL[t])$.

By applying row and column operations, we can transform the matrix $C$ into diagonal form with diagonal entries $\lambda_1,\ldots, \lambda_d$ as in the theorem, due to the existence of left Euclidean and right Euclidean algorithms in $\cL[t]$.

The isomorphism of $\cL[t]$-modules is then obtained by recognizing that $M$ is the cokernel of the map $\cL[t]^d\to\cL[t]^d$ given by $C$ on the standard basis, and row and column operations on $C$ just correspond to changes of bases in the source and target, respectively.
\end{rem}

\begin{thm}\label{thm:decomposition-in-one-slopes}
Given a left $\cL[t]$-module $M$ which is finitely generated as $\cL$-vector space of dimension $d$, 
there exist monic polynomials $f_1,\ldots,f_k\in \cL[t]\smallsetminus\cL$ ($k\leq d$) such that the Newton polygon of each $f_i$ consists of one edge, and
\[  M \cong \cL[t]/f_1\cL[t]\oplus \ldots \oplus\cL[t]/f_k\cL[t] \]
as $\cL[t]$-modules.
\end{thm}

\begin{proof}
This is obtained directly by applying Thm.~\ref{thm:elementary-divisor-theorem}, and then applying Thm.~\ref{thm:chinese-remainder-theorem} to the factors $\cL[t]/\lambda_i\cL[t]$ obtained in Thm.~\ref{thm:elementary-divisor-theorem}.
\end{proof}

We now turn our attention to $\cO$-lattices in such $\cL[t]$-modules $M$.
The next theorem and its corollary is about the existence of a lattice with special properties, whereas Theorem~\ref{thm:one-slope-modules} and Corollary \ref{cor:positive-slopes} are about properties for general $\cO$-lattices.

\begin{thm}\label{thm:stable-O-lattice}
Let $\sigma\in \cL$ be a uniformizer (i.e.~$v(\sigma)=1$), and let $f\in \cL[t]$ be a non-trivial monic polynomial of degree $d$ whose Newton polygon $N_f$ consists of exactly one edge. Let $s\in \QQ$ be the slope of this edge.
Further, let  $M$ be a left $\cL[t]$-module isomorphic to $\cL[t]/f\cL[t]$. 
\begin{enumerate}
\item \label{thm:stable-O-lattice:item:1} There is an $\cO$-lattice $\Lambda$ in $M$ such that
\[  \sigma^{r}t^d\Lambda = \Lambda, \]
where $r=ds\in \NN$. If $s$ is positive, the lattice $\Lambda$ can be chosen to further satisfy $t^{-1}\Lambda\subseteq \Lambda$.
\item \label{thm:stable-O-lattice:item:2} If $u,v\in \ZZ$, $u\neq 0$, and $\Lambda'$ is a lattice in $M$ such that $\sigma^v t^u\Lambda'=\Lambda'$, then $\frac{v}{u}=s$.
\end{enumerate}
\end{thm}

\begin{proof}
Write  $f=\sum_{i=0}^{d-1} c_it^i + t^d$ with $c_i\in \cL$, and let $b\in M$ correspond to the residue class of $1$ in $\cL[t]/f\cL[t]$.
Since the slope of $N_f$ is $s$, we have $v(c_0)=-ds=-r$ and $v(c_i)\geq (i-d)\cdot s$ for $i=1,\ldots, d-1$.
We let $\Lambda$ be the $\cO$-lattice  generated by $\{ \sigma^{\lfloor is\rfloor}t^ib \mid i=0,\ldots, d-1\}$ and claim that this lattice satisfies the desired conditions. Here $\lfloor x \rfloor$ denotes the floor of $x$, i.e.~the largest integer $\leq x$.

We will show $\sigma^{\lfloor ns\rfloor}t^nb\in \Lambda$ for all $n\in \ZZ$, from which one deduces $\sigma^{ds}t^d\Lambda \subseteq \Lambda$ as well as $\Lambda=\sigma^{ds}t^d\left(\sigma^{-ds}t^{-d} \Lambda\right) \subseteq \sigma^{ds}t^d\Lambda$.
If $s$ is positive, one further obtains from this statement that $t^{-1}\sigma^{\lfloor is\rfloor}t^ib \in \Lambda$ for all $i=0,\ldots, d-1$, since $ \lfloor is\rfloor\geq \lfloor (i-1)s\rfloor$ in this case.

\medskip

First at all, the condition $\sigma^{\lfloor ns\rfloor}t^nb\in \Lambda$ is fulfilled for $n=0,\ldots, d-1$ by definition of $\Lambda$.
Let $n>d-1$, then we have
\begin{eqnarray*} 
\sigma^{\lfloor ns\rfloor}t^nb &=& \sigma^{\lfloor ns\rfloor}t^{n-d}\left( -\sum_{i=0}^{d-1} c_it^ib \right) \\
&=&  -\sum_{i=0}^{d-1} \left( \sigma^{\lfloor ns\rfloor}\cdot c_i\right)t^{n-d+i}b.
\end{eqnarray*}
Since $v(c_i)$ is an integer greater or equal to $(i-d)\cdot s$, we have
\[ v\left(\sigma^{\lfloor ns\rfloor}\cdot c_i\right) = \lfloor ns\rfloor + v(c_i)=\lfloor ns+ v(c_i) \rfloor\geq \lfloor (n+i-d)s\rfloor. \]
Hence, by induction hypothesis $\left( \sigma^{\lfloor ns\rfloor}\cdot c_i\right)t^{n-d+i}b\in \Lambda$, and therefore
$\sigma^{\lfloor ns\rfloor}t^nb\in \Lambda$.

On the other hand, for $n<0$, we have
\begin{eqnarray*}
\sigma^{\lfloor ns\rfloor}t^nb 
&=& \sigma^{\lfloor ns\rfloor}t^{n}\left( -c_0^{-1}\cdot \bigl(t^db + \sum_{i=1}^{d-1} c_it^ib\bigr) \right) \\
&=&  -\sigma^{\lfloor ns\rfloor}c_0^{-1} t^{n+d}b -
\sum_{i=1}^{d-1} \left( \sigma^{\lfloor ns\rfloor}\cdot c_0^{-1}\cdot c_i\right)t^{n+i}b.
\end{eqnarray*}
Since $v(c_0)=-ds\in \NN$, we have
\[ v\left(\sigma^{\lfloor ns\rfloor}c_0^{-1}\right)=\lfloor ns\rfloor +ds = \lfloor (n+d)s\rfloor, \]
as well as,
\[ v\left(\sigma^{\lfloor ns\rfloor} \cdot c_0^{-1}\cdot c_i\right)=\lfloor (n+d)s+v(c_i)\rfloor\geq \lfloor (n+d)s+(i-d)s\rfloor
=\lfloor (n+i)s\rfloor. \]
Hence, by backwards induction hypothesis all summands above are in $\Lambda$, and therefore $\sigma^{\lfloor ns\rfloor}t^nb\in \Lambda$.

For showing part \eqref{thm:stable-O-lattice:item:2}, let $u,v\in \ZZ$, $u\neq 0$, and $\Lambda'$ be a lattice in $M$ such that $\sigma^v t^u\Lambda'=\Lambda'$.
Let $\Lambda$ be the lattice constructed in part \eqref{thm:stable-O-lattice:item:1}.
 
As $\Lambda$ and $\Lambda'$ are two $\cO$-lattices in $M$, there are some $l_1,l_2\geq 0$ such that
\[ \sigma^{-l_1}\Lambda'\supseteq \Lambda \supseteq \sigma^{l_2}\Lambda'. \]
Then for any $n\in \ZZ$, we obtain
\begin{eqnarray*}
\Lambda &\supseteq & \sigma^{l_2}\Lambda' =\sigma^{l_2}\sigma^{dnv} t^{dnu}\Lambda' \\
&\supseteq & \sigma^{l_2}\sigma^{dnv} t^{dnu}\sigma^{l_1}\Lambda = \sigma^{l_1+l_2}\sigma^{dnv}\sigma^{-rnu}\Lambda \\
&=& \sigma^{l_1+l_2+n(dv-ru)}\Lambda 
\end{eqnarray*}
If $\frac{v}{u}\neq \frac{r}{d}$, and hence $dv-ru\neq 0$, we obtain a contradiction by choosing $n$ such that $l_1+l_2+n(dv-ru)<0$.
\end{proof}

\begin{cor}\label{cor:stable-O-lattices}
Let $M$ be a left $\cL[t]$-module which is finitely generated as $\cL$-vector space.
Let $f_1,\ldots, f_k\in \cL[t]$ be non-trivial monic polynomials whose Newton polygons consist of exactly one edge such that
$M$ is isomorphic to $\bigoplus_{i=1}^k\cL[t]/f_i\cL[t]$.
For $i=1,\ldots, k$, let $s_i$ be the slope corresponding to $f_i$, $d_i$ the degree of $f_i$, and $M_i\subseteq M$ the submodule corresponding to the factor $\cL[t]/f_i\cL[t]$.

Then for each $i=1,\ldots, k$, there are $\cO$-lattices $\Lambda_i$ in $M_i$ such that 
\[ \sigma^{d_is_i}t^{d_i}\Lambda_i=\Lambda_i. \]
If all the slopes are positive, the lattices can be chosen to further satisfy $t^{-1}\Lambda_i\subseteq \Lambda_i$.
\end{cor}

\begin{proof}
We just have to apply the previous theorem to each factor $M_i$.
\end{proof}

\begin{thm}\label{thm:one-slope-modules}
Let $\sigma\in \cL$ be a uniformizer (i.e.~$v(\sigma)=1$), and let $f\in \cL[t]$ be a non-trivial monic polynomial whose Newton polygon $N_f$ consists of exactly one edge.
Further, let  $M$ be a left $\cL[t]$-module isomorphic to $\cL[t]/f\cL[t]$, and let $\Lambda\subset M$ be an $\cO$-lattice in $M$.
\begin{enumerate}
\item \label{thm:item:1} If the slope of $N_f$ is non-positive, then there exists an $\cO$-basis $(b_1,\ldots, b_d)$ of $\Lambda$ such that for all $k\geq 0$ and $m\in \Lambda$,
\[\pr_d(t^k\cdot m)\in \cO\cdot b_d,\]
where $\pr_d:M\to \cL\cdot b_d,\sum_{i=1}^d x_ib_i\mapsto x_db_d$ denotes the projection to the last coordinate.
\item \label{thm:item:2} If the slope of $N_f$ is positive, there exist $k_0\geq 0$ such that for all $k\geq k_0$,
\[ t^k\Lambda \supseteq \sigma^{-1}\Lambda\]
\end{enumerate}
\end{thm}

\begin{proof}
\begin{enumerate}
\item Since the slope of $N_f$ is non-positive, all the coefficients of $f$ lie in $\cO$. Let $b\in M$ correspond to the residue class of $1$ in $\cL[t]/f\cL[t]$. Let $\Lambda'$
be the $\cO$-lattice of $M$ generated by $b, tb,\ldots, t^{d-1}d$. This is indeed an $\cO[t]$-submodule, since the coefficients of the monic polynomial $f$ lie in $\cO$.

By the elementary divisor theorem (Thm.~\ref{thm:modules-over-pids}\eqref{nj:thm17}), there is a basis $(b_1,\ldots, b_d)$ of $\Lambda$ and $n_1,\ldots, n_d\in \ZZ$ such that $(\sigma^{n_1}b_1,\ldots, \sigma^{n_d}b_d)$ is a basis of $\Lambda'$.\footnote{We take into account that all left ideals of $\cO$ are of the form $\cO\sigma^k$ with $k\in \NN$.}
By rearranging the basis, we can achieve that $n_d$ is the maximum of the numbers $n_i$.

Then for all $k\geq 0$, and $m=\sum_{i=1}^d x_ib_i\in\Lambda$ (i.e. $x_i\in \cO$) we obtain:
\[ t^k\cdot m=\sum_{i=1}^d x_i\cdot t^k\cdot b_i=
\sum_{i=1}^d \sigma^{-n_i}\cdot t^k\cdot \underbrace{\left(\sigma^{n_i}x_i\sigma^{-n_i}\right)}_{\in \cO}\sigma^{n_i}b_i
\in \sigma^{-n_d}\Lambda', \]
since $t^k\cdot \left(\sigma^{n_i}x_i\sigma^{-n_i}\right)\cdot \sigma^{n_i}b_i\in \Lambda'$ for all $i=1,\ldots, d$, and 
$n_d=\max\{ n_i\mid i=1,\ldots, d\}$.

As $\pr_d(\Lambda')=\cO\cdot \sigma^{n_d}b_d$, we finally get:
\[ \pr_d(t^k\cdot m) \in \sigma^{-n_d}\cO\sigma^{n_d}\cdot b_d=\cO\cdot b_d.\]
\item 
Let $\Lambda'$ be an $\cO$-lattice in $M$ as in Theorem \ref{thm:stable-O-lattice}, i.e.~satisfying $\sigma^rt^d\Lambda'=\Lambda'$, as well as $t^{-1}\Lambda'\subseteq \Lambda'$, where $d=\deg_t(f)$, and $\frac{r}{d}=s$ is the slope of the Newton polygon $N_f$.
As $\Lambda$ and $\Lambda'$ are two $\cO$-lattices in $M$, there are some $l_1,l_2\geq 0$ such that
\[ \sigma^{-l_1}\Lambda'\supseteq \Lambda \supseteq \sigma^{l_2}\Lambda'. \]
Let $n_0\in \NN$ satisfy $n_0\geq \frac{1+l_1+l_2}{r}$, and let $k_0=dn_0$. We then obtain for all $k\geq k_0$,
\begin{eqnarray*}
t^k\Lambda &\supseteq & t^k\sigma^{l_2}\Lambda' =  \sigma^{l_2}t^{dn_0}t^{k-k_0}\Lambda' \\
&\supseteq & \sigma^{l_2}t^{dn_0}\Lambda' = \sigma^{l_2}\sigma^{-rn_0}\Lambda' \\
%&=& \sigma^{l_2}t^{k-k_0}\sigma^{-rn_0}\Lambda' \supseteq \sigma^{l_2}\sigma^{rn_0}\Lambda' \\
&\supseteq & \sigma^{l_2}\sigma^{-rn_0}\sigma^{l_1}\Lambda \supseteq \sigma^{-1}\Lambda.
\end{eqnarray*}

\end{enumerate}
\end{proof}

\begin{cor}\label{cor:positive-slopes}
Let $\sigma\in \cL$ be a uniformizer, and let $g_1,\ldots, g_k\in \cL[t]$ be non-trivial monic polynomials whose Newton polygons only have positive slopes.
Further, let  $M$ be a left $\cL[t]$-module isomorphic to $\bigoplus_{i=1}^k\cL[t]/g_i\cL[t]$, and let $\Lambda\subset M$ be an $\cO$-lattice in $M$.
Then there is a natural number $n_0\in \NN$, such that for all $n\geq n_0$,
\[  t^n\Lambda\supseteq \sigma^{-1}\Lambda. \]
\end{cor}

\begin{proof}
By Thm.~\ref{thm:decomposition-in-one-slopes}, we can decompose the factors $\cL[t]/g_i\cL[t]$ further to obtain a decomposition 
$M\cong \bigoplus_{j=1}^r \cL[t]/f_j\cL[t]$ where the Newton polygons of all $f_j$ have exactly one edge, and the slope of that edge is positive.
 
Let $\Lambda^{(j)}:=\Lambda\cap \cL[t]/f_j\cL[t]$ for all $j=1,\ldots, r$, and 
\[\Lambda':=\bigoplus_{j=1}^r \Lambda^{(j)}\subseteq \Lambda.\]
Applying Thm.~\ref{thm:one-slope-modules}\eqref{thm:item:2} to each factor, we obtain numbers $n_1,\ldots,n_r\in \NN$ such that for each $i=1,\ldots,r$, one has $t^{n}\Lambda^{(i)}\supseteq \sigma^{-1}\Lambda^{(i)}$ for $n\geq n_i$, and hence
$n':=\max\{n_1,\ldots, n_r\}$ fulfills
\[ t^{n}\Lambda'\supseteq  \sigma^{-1}\Lambda' \quad \forall\, n\geq n'.\]
As $\Lambda'$ and $\Lambda$ are two lattices in $M$, we can proceed similar to the end of the proof of Thm.~\ref{thm:one-slope-modules} \eqref{thm:item:2} to obtain some $n_0\in \NN$ such that for all $n\geq n_0$,
\[ t^{n}\Lambda\supseteq  \sigma^{-1}\Lambda. \qedhere \]
\end{proof}

\section{Matrices over complete discretely valued skew fields}\label{sec:matrices}

Let $\cL$ be a complete discretely valued skew field, and $\cO$ its valuation ring with uniformizer~$\sigma$. Let $d\geq 1$, and $D\in \Mat_{d\times d}(\cL)$ a nonzero matrix. Further, we let $\lambda_1,\ldots,\lambda_d\in \cL[t]$ be its invariant factors, i.e.~polynomials $\lambda_1|\lambda_2|\ldots|\lambda_d$ obtained by diagonalizing the matrix $t\cdot \one_d-D$ as in Remark \ref{rem:invariant-factor-computation}.

%
%Furthermore, let $\Lambda$ be a free left $\cO$-module of rank $d$, and $M=\cL\otimes_{\cO} \Lambda$.
%
%Throughout this section, $\phi:M\to M$ is a fixed endomorphism of $\cL$-modules. Given a $\cL$-basis $(e_1,\ldots,e_d)$ of $M$, this endomorphism can be represented by a matrix $D\in \Mat_{d\times d}(\cL)$ which is defined by
%\[ \phi(e_i)= \sum_{j=1}^d D_{ij} e_j, \quad \text{for all }i=1,\ldots, d, \]
%or in matrix form
%\[ \phi \svect{e}{d} = D\cdot \svect{e}{d}. \]
%
%We further consider $M$ as a module over the polynomial ring $\cL[t]$ where $t$ acts by applying~$\phi$.

In this section, we apply the results of the previous sections to obtain the equivalence of the following conditions on $D$ and its powers which will be used in the next sections.

\begin{thm}\label{thm:equvialent-matrix-conditions}
Let $D\in \Mat_{d\times d}(\cL)$, and let $\lambda_d\in \cL[t]$ be its last invariant factor as above. For $n\geq 1$, we define $s_n\in \NN$ to be the least non-negative integer such that all entries of $D,D^2,\ldots, D^n$ have valuation greater or equal to $-s_n$. The following are equivalent:
\begin{enumerate}
\item \label{item:1} There is some $n\geq 1$ such that the matrix $\sigma^{s_n}\cdot D^n\in \Mat_{d\times d}(\cO)$ has full rank $d$ modulo $\sigma^{s_n}$,\footnote{See Definition \ref{def:rank-modulo-sigma-s} for our notion of rank modulo some power of $\sigma$.}
\item \label{item:2} there is some $n\geq 1$ such that the block matrix 
$\begin{smatrix}\sigma^{s_n}\cdot D\\ \sigma^{s_n}\cdot D^2\\ \vdots \\ \sigma^{s_n}\cdot D^n\end{smatrix} 
 \in \Mat_{nd\times d}(\cO)$ has rank $d$ modulo $\sigma^{s_n}$,
\item[(2')] \label{item:2prime} there is some $n\geq 1$ such that the block matrix 
$\begin{pmatrix}\sigma^{s_n}\cdot D, & \sigma^{s_n}\cdot D^2,& \ldots,& \sigma^{s_n}\cdot D^n\end{pmatrix} 
 \in \Mat_{d\times nd}(\cO)$ has rank $d$ modulo $\sigma^{s_n}$,
\item \label{item:3} all slopes of the Newton polygon of $\lambda_d$ are positive.
\end{enumerate}
\end{thm}

\begin{rem}
In the first three conditions, we could replace $s_n$ by any larger number. This would lead to an equivalent condition. This fact becomes clear during the proof of the theorem.
\end{rem}

\begin{proof}
The implications \eqref{item:1}$\Rightarrow$\eqref{item:2} and \eqref{item:1}$\Rightarrow$
%\eqref{item:2prime}
(2') are trivial. We show \eqref{item:2}$\Rightarrow$\eqref{item:3} by contraposition. So assume that the Newton polygon of $\lambda_d$ has an edge of non-positive slope.

We let $\Lambda$ be a free $\cO$-module of rank $d$, and $(e_1,\ldots,e_d)$ be a basis. Further, let 
$M=\cL\otimes_{\cO} \Lambda$,
and define an $\cL$-linear $t$-action on it by letting
\[ t\cdot  \svect{e}{d} = D\cdot \svect{e}{d} \]
component-wise. This turns $M$ into a left-$\cL[t]$-module.
By Thm.~\ref{thm:elementary-divisor-theorem}, $M$ can be decomposed into a direct sum of $\cL[t]$-modules $\bigoplus_{i=1}^d \cL[t]/\lambda_i\cL[t]$ where the $\lambda_i$ are the invariant factors of $D$.
By Thm.~\ref{thm:decomposition-in-one-slopes}, we can decompose these further into
\[ M \cong \bigoplus_{i=1}^r \cL[t]/f_i\cL[t], \]
where the Newton polygons of each $f_i$ has exactly one edge. Without loss of generality, we can assume
that $f_r$ is a factor of $\lambda_d$ whose Newton polygon consists of one edge with non-positive slope.
Let $M'\subset M$ be the submodule corresponding to the sum $\bigoplus_{i=1}^{r-1} \cL[t]/f_i\cL[t]$, and $\Lambda':= \Lambda\cap M'$. Then by Prop.~\ref{prop:integral-submodule}, the quotient $\Lambda'':=\Lambda/\Lambda'$ is free, and $\cL\otimes_{\cO} \Lambda''$ has to be isomorphic to the last factor $\cL[t]/f_r\cL[t]$ of $M$.

We choose a basis $(b_1,\ldots,b_m,b_{m+1},\ldots, b_d)$ of $\Lambda$ in the following way: $(b_1,\ldots,b_m)$ is a basis of $\Lambda'$, and $(b_{m+1},\ldots, b_d)$ is a lift of a basis $(\bar{b}_{m+1},\ldots, \bar{b}_d)$ of $\Lambda''$ where $(\bar{b}_{m+1},\ldots, \bar{b}_d)$ is chosen as in Thm.~\ref{thm:one-slope-modules}\eqref{thm:item:1}. By this choice of basis, as in Thm.~\ref{thm:one-slope-modules}, we obtain $\pr_d(t^k\cdot m)\in \cO\cdot b_d$ for all $k\geq 0$ and all $m\in \Lambda$ where $\pr_d:M\to \cL\cdot b_d$ denotes the projection to the last coordinate.

If $C\in \Mat_{d\times d}(\cL)$ is the matrix representing the $t$-action on this basis, i.e. given by
\[ t\cdot  \svect{b}{d} = C\cdot \svect{b}{d}, \]
this implies that the last column of each $C^k$ ($k\geq 0$) has entries in $\cO$. 
In particular, for $n\geq 1$, and for $s\geq 0$ such that all entries of all $C, C^2\ldots, C^n$ have valuation at least $-s$,
the entries of the last column of 
\[ \begin{smatrix}\sigma^{s}\cdot C\\ \sigma^{s}\cdot C^2\\ \vdots \\ \sigma^{s}\cdot C^n\end{smatrix}\in \Mat_{nd\times d}(\cO) \]
are divisible by $\sigma^{s}$. This implies that the rank modulo $\sigma^{s}$ of this block matrix does not exceed $d-1$.

As $(e_1,\ldots, e_d)$ and $(b_1,\ldots, b_d)$ are two bases of $\Lambda$, there exists a unique matrix $B\in \GL_d(\cO)$ such that
\[ \svect{b}{d} = B\cdot \svect{e}{d}, \]
and hence, $C=BDB^{-1}$. On one hand, this means that $s$ above can be chosen to be $s_n$, on the other hand that the rank modulo $\sigma^{s}$ of
\[ \begin{smatrix}\sigma^{s}\cdot D\\ \sigma^{s}\cdot D^2\\ \vdots \\ \sigma^{s}\cdot D^n\end{smatrix}\in \Mat_{nd\times d}(\cO) \]
does not exceed $d-1$, since we have
\[ \begin{smatrix}\sigma^{s}\cdot D\\ \sigma^{s}\cdot D^2\\ \vdots \\ \sigma^{s}\cdot D^n\end{smatrix}
= \scalar{\sigma^{s}B^{-1}\sigma^{-s}}\cdot \begin{smatrix}\sigma^{s}\cdot C\\ \sigma^{s}\cdot C^2\\ \vdots \\ \sigma^{s}\cdot C^n\end{smatrix}\cdot B\in \Mat_{nd\times d}(\cO).\]

The implication %\eqref{item:2prime}
(2')$\Rightarrow$\eqref{item:3} is shown in the same manner, but using right $\cO$-modules.

It remains to show the implication \eqref{item:3}$\Rightarrow$\eqref{item:1}. So we assume that all slopes of the Newton polygon of $\lambda_d$ are positive. As the other invariant factors are divisors of $\lambda_d$ this implies that all slopes of all the Newton polygons are positive.

As above, we consider the free $\cO$-module $\Lambda$ of rank $d$ with basis $(e_1,\ldots,e_d)$,
and the $\cL$-vector space $M=\cL\otimes_{\cO} \Lambda$ with additional $t$-action given by
\[ t\cdot  \svect{e}{d} = D\cdot \svect{e}{d}. \]
Again by Thm.~\ref{thm:decomposition-in-one-slopes}, we can decompose $M$ into
\[ M \cong \bigoplus_{i=1}^r \cL[t]/f_i\cL[t], \]
where the Newton polygons of each $f_i$ has exactly one edge, but this time all slopes are positive.
By Cor.~\ref{cor:positive-slopes}, there is some $n\in \NN$ such that 
$t^n\Lambda\supseteq \sigma^{-1}\Lambda$.

Hence, there is a matrix $B\in \Mat_{d\times d}(\cO)$ such that
\[ BD^n\cdot \svect{e}{d} = B\cdot t^n \svect{e}{d}= \sigma^{-1}\svect{e}{d} =\scalar{ \sigma^{-1}}
\cdot \svect{e}{d}.\]
Therefore we have
\[ \left(\sigma^{s_n}B\sigma^{-s_n}\right)\cdot \sigma^{s_n}D^n = \scalar{\sigma^{s_n-1}}\in \Mat_{d\times d}(\cO).\]
As $\sigma^{s_n}B\sigma^{-s_n}\in \Mat_{d\times d}(\cO)$, and the right hand side has full rank $d$ modulo $\sigma^{s_n}$, also $\sigma^{s_n}D^n\in \Mat_{d\times d}(\cO)$ has full rank $d$ modulo $\sigma^{s_n}$.
\end{proof}

\section{$t$-modules and $t$-motives} \label{sec:t-modules}

From now on, let $\Fq$ be the finite field with $q$ elements, and let 
%$p$ be its characteristic. Let 
$K$ be a perfect field containing $\Fq$.
Further, let $\Fq[t]$ be a polynomial ring over $\Fq$ in an indeterminate $t$ which is linearly independent to $K$, and $\ell:\Fq[t]\to K$ a homomorphism of $\Fq$-algebras.

%$\e:\Lie(E)(\CC_\infty)\to E(\infty)$ exponential map, is unique with the properties 
%\[\forall a\in A, x\in\Lie(E)(\CC_\infty):\ \phi_a(\e(x))=\e(\dphi_a(x)) \]
%and $d(\e)=\id_{\Lie(E)}$

As in the introduction, we denote by $K\{\tau\}$ the skew polynomial ring
\[  K\sp{\tau}=\left\{ \sum_{i=0}^n \alpha_i\tau^i \,\middle|\, n\geq 0, \alpha_i\in K\right\} \]
with multiplication uniquely given by additivity and the rule
\[ \tau \cdot \alpha = \alpha^q\cdot \tau, \]
for all $\alpha\in K$, i.e., 
\[ \left( \sum_{i=0}^n \alpha_i\tau^i\right) \cdot \left(\sum_{j=0}^m \beta_j\tau^j\right)
= \sum_{k=0}^{n+m} \left( \sum_{i=0}^k \alpha_i\cdot (\beta_{k-i})^{q^i} \right) \cdot \tau^k.  \]
This ring equals the ring $\End_{\grp,\Fq}(\Ga)$ of $\Fq$-linear group endomorphisms of $\Ga$ by identifying $\tau$ with the $q$-th power Frobenius map, and $\alpha\in K$ with scalar multiplication by~$\alpha$.

We further consider the skew Laurent series ring over $K$ in $\sigma=\tau^{-1}$,
\[ K\sls{\sigma} :=\left\{ \sum_{i=i_0}^\infty \alpha_i\sigma^i \,\middle|\, i_0\in \ZZ, \alpha_i\in K\right\}, \]
with $\sigma \cdot \alpha = \alpha^{1/q}\cdot \sigma$ for all $\alpha\in K$ (well defined as $K$ is assumed to be perfect), as well as the subring of skew power series in $\sigma$,
\[ K\sps{\sigma} :=\left\{ \sum_{i=0}^\infty \alpha_i\sigma^i \,\middle|\, \alpha_i\in K\right\}. \]

We equip the ring $K\sls{\sigma}$ with the discrete valuation $v$ given by
\[ v\left( \sum_{i=i_0}^\infty \alpha_i\sigma^i\right) = \inf\{ i\in \ZZ \mid \alpha_i\ne 0\}, \]
i.e., the order of the series in $\sigma$. This turns $K\sls{\sigma}$ into a complete discretely valued skew field with valuation ring  $K\sps{\sigma}$ and uniformizer $\sigma$.
The ring $K\sp{\tau}$ is naturally embedded into $K\sls{\sigma}$ via
\[ \sum_{i=0}^n \alpha_i\tau^i \mapsto \sum_{i=0}^n \alpha_i\sigma^{-i}. \]

For $x=\sum_{i=i_0}^\infty \alpha_i\sigma^i\in K\sls{\sigma}$, and $k\in \ZZ$, we set its $k$-th  twist to be 
\[ x^{(k)}:=\sum_{i=i_0}^\infty \alpha_i^{q^k}\sigma^i.\]
Also for matrices $B\in \Mat_{d\times e}(K\sls{\sigma})$, and $k\in \ZZ$, we denote by
$B^{(k)}$ the matrix whose $(i,j)$-th entry is the $k$-th twist of the $(i,j)$-th entry of $B$.

For a matrix $B\in \Mat_{d\times e}(K\sls{\sigma})$, we write $v(B)$ for the infimum of all valuations $v(B_{i,j})$ of entries of $B$.

\medskip

A $t$-module $(E,\phi)$ over $K$ of dimension $d$ is by definition an $\Fq$-vector space scheme $E$ over $K$ isomorphic to $\Ga^d$ together with a homomorphism of $\Fq$-algebras $\phi:\Fq[t]\to \End_{\grp,\Fq}(E)$ into the ring of $\Fq$-vector space scheme endomorphisms of $E$, such that for all $a\in \Fq[t]$, the endomorphism $\dphi_a$ on $\Lie(E)$ induced by $\phi_a$ fulfills the condition that $\dphi_a-\ell(a)$ is nilpotent.

Throughout this and the next sections, we fix a $t$-module $(E,\phi)$  over $K$ of dimension $d$, as well as a coodinate system $\kappa$, i.e., an isomorphism of $\Fq$-vector space schemes $\kappa:E\cong \Ga^d$ defined over $K$.

With respect to this coordinate system, we can represent $\phi_t$ by a matrix $D\in \Mat_{d\times d}(K\{\tau\})$. Formally, the endomorphism $\tilde{\phi}_t:=\kappa \circ \phi_t\circ \kappa^{-1}$ is given as
\[ \tilde{\phi}_t\svect{x}{d} = D\cdot \svect{x}{d}, \]
for all $\transp{(x_1,\ldots, x_d)}\in \Ga^d(K)$.

Later we will use the maximal $\tau$-degree of entries of $D$, and we will shortly write $\deg_{\tau}(\phi_t)$ for this number.\footnote{Be aware that our notation is a bit lazy, as that number depends on the chosen coordinate system in general. However, this doesn't cause any trouble, as we fix one coordinate system throughout the paper.} 
In the same way for any $a\in \Fq[t]\smallsetminus \Fq$, we denote by $\deg_{\tau}(\phi_a)$ the maximal $\tau$-degree of entries of the matrix representing $\tilde{\phi}_a$.

Let $\kappa_i:E\to \Ga$ be the $i$-th component of $\kappa$, i.e.~the composition of $\kappa$ with the projection $\pr_i:\Ga^d\to \Ga$ to the $i$-th component of $\Ga^d$. 
Then the tuple $(\kappa_1,\ldots, \kappa_d)$ is a $K\{\tau\}$-basis of the $t$-motive of $E$,
\[ \mot:=\Hom_{\grp,\Fq}(E,\Ga). \]
The dual $t$-motive associated to $E$ is
\[ \dumot:=\Hom_{\grp,\Fq}(\Ga,E). \]
It is usually considered as a left $K\{\sigma\}[t]$-module, where the $t$-action is given by composition with $\phi_t$, and the left-$K\{\sigma\}$-action stems from the natural right action of $K\{\tau\}=\End_{\grp,\Fq}(\Ga)$, by considering $K\{\sigma\}$ as the opposite ring of $K\{\tau\}$.
In this paper, however, we stick to considering $\dumot$ with the natural right-$K\{\tau\}$-action, and will also write the $t$-action from the right. 
Given the choice of coordinate system $\kappa$ above, a $K\{\tau\}$-basis of  $\dumot$ is given by
$(\dk_1,\ldots, \dk_d)$, where $\dk_j:\Ga\to E$ is the composition of the injection $\inj_j:\Ga\to \Ga^d$ into the $j$-th component with $\kappa^{-1}$ for all $j=1,\ldots, d$. The situation is depicted in the following commutative diagram.
\[ \xymatrix{
 & & \Ga \\
 E \ar[r]^{\kappa}_{\cong} \ar@/^0.5pc/[rru]^{\kappa_i} & \Ga^d \ar[ru]_{\pr_i} & \\
 & & \Ga \ar[ul]_{\inj_j} \ar[uu]_{\delta_{ij}\cdot \id_{\Ga}} \ar@/^0.5pc/[ull]^{\dk_j}
} \]
%We also remark that the identity
%\[ \sum_{j=1}^d \dk_j\circ \kappa_j = \id_E \]
%holds.
A short calculation shows that via these bases, the $t$-action on the $t$-motive is described by %\footnote{Is that common knowledge? Student exercise? Or should I prove it?}
\begin{equation} \label{eq:t-action-on-mot} 
t\cdot \svect{\kappa}{d} = D\cdot \svect{\kappa}{d}, 
\end{equation}
and the $t$-action on the dual $t$-motive is described by
\begin{equation}\label{eq:t-action-on-dumot} 
 \zvect{\dk}{d}\cdot t = \zvect{\dk}{d}\cdot D.
\end{equation}

We will usually consider $D$ as a matrix with coefficients in $K\sls{\sigma}\supseteq K\sp{\tau}$, and recognize that the maximum of the $\tau$-degrees of the entries of $D$ in $K\sp{\tau}$ is nothing else than the additive inverse of the valuation $v(D)$ of $D\in K\sls{\sigma}$.

Therefore, when $s$ is the maximal $\tau$-degree of entries of $D$, the matrix $\sigma^sD\in \Mat_{d\times d}(K\sls{\sigma})$ has entries of non-negative valuation, hence $\sigma^sD\in \Mat_{d\times d}(K\sps{\sigma})$, and even in $\Mat_{d\times d}(K\{\sigma\})$.

\section{Criterion for abelian and $t$-finite $t$-modules}\label{sec:criterion}

In this section, we prove our main theorem \ref{mainthm:equivalence} on abelian and $t$-finite $t$-modules.

We use the notion of the previous section.

\begin{prop}\label{prop:sufficient-criterion}
Assume that there exists $a\in\Fq[t]\smallsetminus \Fq$, $s:=\deg_{\tau}(\phi_a)$ such that the matrix  representing $\sigma^s\cdot \phi_a$  has full rank modulo $\sigma^{s}$. Then the following hold.
\begin{enumerate}
\item \label{prop:sufficient-criterion:item:1} $E$ is abelian, i.e.~$\mot$ is a finitely generated $K[t]$-module.
\item \label{prop:sufficient-criterion:item:2} $E$ is $t$-finite, i.e.~$\dumot$ is a finitely generated $K[t]$-module.
\end{enumerate}
\end{prop}

\begin{proof}
We start by proving \eqref{prop:sufficient-criterion:item:1}, and
we first consider the case that the hypothesis is fulfilled for $a=t$.

With respect to the $K\{\tau\}$-basis $\kappa_1,\ldots,\kappa_d$ of $\mot$, we have
\[ t\cdot \svect{\kappa}{d} = D\cdot \svect{\kappa}{d}. \]
Hence, the matrix $D-t\cdot \one_d\in \Mat_{d\times d}(K[t]\{\tau\})$ annihilates the $K\{\tau\}$-basis $\kappa_1,\ldots,\kappa_d$.

Write $D=D_0+D_1\tau+\ldots + D_s\tau^s$ with $D_i\in \Mat_{d\times d}(K)$ ($i=0,\ldots, s$), and $D_s\neq 0$ (as $s=\deg_{\tau}(\phi_t)$).
If the matrix $D_s$ is invertible, %we are done. Indeed, this implies that 
we can write
\[ \tau^s \svect{\kappa}{d}=-D_s^{-1}\cdot \left( D-t\cdot \one_d-D_s\tau^s\right)\cdot \svect{\kappa}{d}, \]
and the $\tau$-degree of $D-t\cdot \one_d-D_s\tau^s$ is at most $s-1$.
Therefore all $\tau^s\kappa_j$ are $K[t]$-linear combinations of the $\tau^i\kappa_j$ with $i<s$ and $j=1,\ldots, d$, and by twisting the equation by powers of $\tau$, we obtain, that all $\tau^k\kappa_j$
with $k\geq s$ are in the $K[t]$-span of the $\tau^i\kappa_j$ with $i<s$ and $j=1,\ldots, d$. In particular, $\mot$ is finitely generated as $K[t]$-module.\footnote{Actually, unless the $t$-module $E$ wasn't a product of $\GG_a$'s with trivial $t$-action, the considered case with $D_s\in \GL_d(K)$ is called a \emph{strictly pure} $t$-module in \cite{cn-mp:hpqpam}, and it is well known to be abelian.}

If $D_s$ is not invertible, we are going to find a matrix $D'\in \Mat_{d\times d}(K\{\tau\})$ such that 
$D'\cdot (D-t\one_d)$ has the property that its top $\tau$-coefficient matrix is an invertible matrix with
coefficients in $K$.
Then we can conclude the finite generation as in the special case.

Consider the matrix $\sigma^s D\in \Mat_{d\times d}(K\{\sigma\})$. As $K\{\sigma\}$ is a skew PID, there are matrices $B,C\in \GL_{d}(K\{\sigma\})$ such that $B(\sigma^s D)C$ is a diagonal matrix (cf.~Thm.~\ref{thm:modules-over-pids}\eqref{nj:thm16}). As by assumption on $D$, the rank of $\sigma^s D$ modulo $\sigma^s$ is $d$, all the diagonal entries of this matrix have $\sigma$-orders less then $s$. Let $\nu_1,\ldots, \nu_d\in \{0,\ldots, s-1\}$ be these orders.

 Multiplying from the left with the diagonal matrix 
\[ T =  \begin{pmatrix} 
\tau^{\nu_1} & 0 & \cdots & 0\\ 
0 & \tau^{\nu_2} & \ddots & \vdots\\
\vdots & \ddots & \ddots & 0 \\
0 & \cdots & 0 & \tau^{\nu_d} \end{pmatrix}, \] 
we get $TB(\sigma^s D)C\in \Mat_{d\times d}(K\{\sigma\})$, and
\[ TB(\sigma^s D)C \equiv S \mod \sigma\] with $S\in \GL_d(K)$ (of course a diagonal matrix). Hence,
\[ CTB(\sigma^s D) \equiv CSC^{-1}\equiv S' \mod \sigma\] with $S'\in \GL_d(K)$.
As $\nu_i<s$ for all $i=1,\ldots, d$, one further has
\[ CTB(\sigma^s\cdot t\one_d) \in \sigma\Mat_{d\times d}(K[t]\{\sigma\}).\]

Finally, there is an integer $r\in \NN$ such that $D':=\tau^{r}CTB\sigma^s\in \Mat_{d\times d}(K\{\tau\})$, i.e.~that no negative $\tau$-powers remain.
By construction, this $D'$ satisfies the desired property, as the top $\tau$-coefficient matrix of
$D'\cdot (D-t\one_d)$
is the $r$-th twist $S'^{(r)}$ of $S'\in \GL_d(K)$.

In the general case, i.e.~that the hypothesis holds for some $a\in \Fq[t]\smallsetminus \Fq$, we set $D$ to be the matrix representing $\phi_a$. The same proof as above shows that the $t$-motive $\mot$ is finitely generated as $\Fq[a]$-module. In particular, $\mot$ is finitely generated as $\Fq[t]$-module, hence $E$ is abelian.

\medskip

The proof for the dual $t$-motive is almost identical to the previous one, but with sides swapped. We briefly sketch the main steps, pointing out similarities and differences. The general case when the hypothesis is fulfilled for some $a$ is obtained from the special case $a=t$ in the same way. So we restrict to the case that the hypothesis on the rank holds for the matrix $\sigma^s D$ representing $\sigma^s \cdot \phi_t$. 
As explained before, the $t$-action on $\dumot$ is given by
\[ \zvect{\dk}{d}\cdot t = \zvect{\dk}{d}\cdot D,\]
so $D-t\one_d\in \Mat_{d\times d}(K[t]\{\tau\})$ is annihilating the basis $(\dk_1,\ldots, \dk_d)$ (but this time by multiplication from the right).

If the top coefficient $D_s$ of $D$ is invertible, then so is its $(-s)$-th twist $D_s^{(-s)}$, and by using
\[ D_s\tau^s=\tau^s D_s^{(-s)}, \]
we get 
\[ \zvect{\dk}{d}\tau^s=-\zvect{\dk}{d}\cdot \left( D-t\cdot \one_d-D_s\tau^s\right)\cdot \left(D_s^{(-s)}\right)^{-1}, \]
and we deduce the finite generation of $\dumot$ as $K[t]$-module in the same way as for $\mot$.

If $D_s$ is not invertible, we are going to find a matrix $D'\in \Mat_{d\times d}(K\{\tau\})$ such that 
$(D-t\one_d)\cdot D'$ has the property that its top $\tau$-coefficient matrix is an invertible matrix with
coefficients in $K$, in order to conclude finite generation as in the special case.

First at all, we have the same matrices $B,C\in \GL_{d}(K\{\sigma\})$ as in the proof for the $t$-motive such that $B(\sigma^s D)C\in \Mat_{d\times d}(K\{\sigma\})$ is a diagonal matrix, and we can choose the same matrix $T$, in order to obtain that $B(\sigma^s D)CT\in \Mat_{d\times d}(K\{\sigma\})$ and
\[ B(\sigma^s D)CT \equiv \tilde{S} \mod \sigma\] with $\tilde{S}\in \GL_d(K)$ (of course a diagonal matrix). Hence,
\[ (\sigma^sD)\cdot CTB \equiv  B^{-1}\tilde{S}B\equiv \tilde{S}' \mod \sigma\] with $\tilde{S}'\in \GL_d(K)$,
and 
\[\tau^s(\sigma^sD)\cdot CTB\sigma^s=\left((\sigma^sD)\cdot CTB\right)^{(s)}\equiv\tilde{S}'^{(s)}\mod \sigma.\]
Further, as above
\[ t\one_d\cdot CTB\sigma^s \in \sigma\Mat_{d\times d}(K[t]\{\sigma\}).\]

Finally, there is $r\in \NN$ such that $D':=CTB\sigma^s\tau^{r}\in \Mat_{d\times d}(K\{\tau\})$, i.e.~that no negative $\tau$-powers remain.
By construction, this $D'$ satisfies the desired property, as the top $\tau$-coefficient matrix of
$(D-t\one_d)\cdot D'$
is $\tilde{S}'^{(s)}\in \GL_d(K)$.
\end{proof}

\begin{exmp}
The almost strictly pure $t$-modules defined in \cite[\S 4.5]{cn-mp:hpqpam} are examples of $t$-modules satisfying the hypothesis. Indeed, by definition these are the ones where for some $n\in\NN$, the leading coefficient matrix of the matrix representing $\tilde{\phi}_{t^n}$ is invertible.
\end{exmp}

Next, we provide an example of a simple $t$-module of dimension $2$ which is abelian, but (at least in characteristic different from $2$) is not almost strictly pure.
This shows that there are more abelian $t$-modules than almost strictly pure ones and extensions of those.

\begin{exmp}\label{exmp:new-t-module}
Consider the $t$-module $(E,\phi)$ over the rational function field $K=\Fq(\theta)$ with
\[ \phi_t=\begin{pmatrix} \theta & 0 \\ 1 & \theta \end{pmatrix}+
\begin{pmatrix} 0 & 0 \\ 1 &0\end{pmatrix}\cdot \tau+
\begin{pmatrix} 1 & 0 \\ 0 &1\end{pmatrix}\cdot \tau^2+
\begin{pmatrix} 0 & 1 \\ 0 &0\end{pmatrix}\cdot \tau^3 =
\begin{pmatrix} \theta+\tau^2 & \tau^3 \\ 1+\tau & \theta+\tau^2 \end{pmatrix}
. \]
The matrix 
\[ \sigma^3\cdot \phi_t=\begin{pmatrix} \sigma+\theta^{1/q^3}\sigma^{3} & 1 \\ \sigma^{2} +\sigma^{3}  &  \sigma+\theta^{1/q^3}\sigma^{3}\end{pmatrix} \]
does not have full rank modulo $\sigma^3$, as modulo $\sigma^{3}$ the second row is the $\sigma$-multiple of the first one.

However, this $t$-module satisfies the hypothesis of the previous theorem with $a=t^2$:\\
\[ \phi_{t^2}=(\phi_t)^2=\begin{pmatrix} \theta+(\theta^{q^2}+\theta)\tau^2+\tau^3+2\tau^4 & 
 (\theta^{q^3}+\theta)\cdot \tau^3+2\tau^5 \\ 
 2\theta+(\theta^q+\theta)\tau+2\tau^2+2\tau^3 & 
 \theta^2+(\theta^{q^2}+\theta)\tau^2+\tau^3+2\tau^4 \end{pmatrix}, \]
and 
\[ \sigma^{5}\phi_{t^2}=\begin{pmatrix} 
2\sigma+\sigma^{2}+(\theta^{1/q^3}+\theta^{1/q^5})\sigma^{3}+\theta^{1/q^5}\sigma^{5} & 
2 +  (\theta^{1/q^2}+\theta^{1/q^5})\cdot \sigma^{2} \\ 
2\sigma^{2}+2\sigma^{3}+(\theta^{1/q^4}+\theta^{1/q^5})\sigma^{4}+2\theta^{1/q^5}\sigma^{5} & 
 2\sigma+\sigma^{2}+(\theta^{1/q^3}+\theta^{1/q^5})\sigma^{3}+\theta^{1/q^5}\sigma^{5} \end{pmatrix}. \]

Therefore, this $t$-module $E$ is abelian and $t$-finite.

We will see later (see Example \ref{exmp:new-t-module-revisited}) that unless the characteristic of $\Fq$ is $2$, the $t$-module $E$ is not pure, and in particular not almost strictly pure (cf.~\cite[\S 4.5]{cn-mp:hpqpam}).%
\footnote{Be aware that the condition of the leading coefficient matrix being invertible depends on the chosen coordinate system. Hence, we can not see this directly from the given matrix for $\phi_t$.}

\medskip

We conclude this example by showing that $E$ is indeed a simple $t$-module:\\
A non-trivial submodule of $E$ would be generated by an eigenvector $\left(\begin{smallmatrix} x\\ y\end{smallmatrix}\right)\in K\sp{\tau}^2$ of the matrix 
\[ \phi_t=\begin{pmatrix} \theta+\tau^2 & \tau^3 \\ 1+\tau & \theta+\tau^2 \end{pmatrix}.\]
Hence, there would be some $c\in K\sp{\tau}$ such that the matrix
\[ \phi_t-c\one_2= \begin{pmatrix} \theta+\tau^2-c & \tau^3 \\ 1+\tau & \theta+\tau^2-c \end{pmatrix}\]
does not have full rank.
By dividing (from the right) by the lowest $\tau$-power occurring in $x$ and $y$, we can assume that
the vector $\transp{(x,y)}$ is not congruent to $\transp{(0,0)}$ modulo $\tau$.

As the matrix $\phi_t-c\one_2$ is lower diagonal modulo $\tau^3$, we must have $c=\theta+\tau^2-\tau^3\cdot c_3$ for some $c_3\in K\sp{\tau}$ which results in the equation
\[ \begin{pmatrix} \tau^3c_3 & \tau^3 \\ 1+\tau & \tau^3c_3 \end{pmatrix}\cdot \begin{pmatrix}  x\\ y\end{pmatrix} = \begin{pmatrix}  0\\ 0\end{pmatrix}.\]
Therefore (considering the equation modulo $\tau^3$ again), $x$ is divisible by $\tau^3$, i.e.~$x=\tau^3x_3$ for some $x_3\in K\sp{\tau}$. Hence,
\begin{eqnarray*} 
\begin{pmatrix}  0\\ 0\end{pmatrix} &=& \begin{pmatrix} \tau^3c_3 & \tau^3 \\ 1+\tau & \tau^3c_3 \end{pmatrix}\cdot \begin{pmatrix}  \tau^3 x_3\\ y\end{pmatrix} = \begin{pmatrix} \tau^3c_3\tau^3 & \tau^3 \\ (1+\tau)\tau^3 & \tau^3c_3 \end{pmatrix}\cdot \begin{pmatrix}  x_3\\ y\end{pmatrix}\\
&=& \tau^3\cdot \begin{pmatrix} c_3\tau^3 & 1 \\ (1+\tau) & c_3 \end{pmatrix}\cdot \begin{pmatrix}  x_3\\ y\end{pmatrix}.
\end{eqnarray*}
As the matrix $\begin{pmatrix} c_3\tau^3 & 1 \\ (1+\tau) & c_3 \end{pmatrix}$ is invertible modulo $\tau^3$, 
we conclude $x_3\equiv y\equiv 0$ modulo $\tau^3$, contradicting the condition that the eiqenvector $\transp{(x,y)}$ is not congruent to $\transp{(0,0)}$ modulo~$\tau$. 

\end{exmp}

\begin{prop}\label{prop:necessary-criterion-mot}
Let $E$ be a $t$-module given as above, and assume that $E$ is abelian. Then there exists a natural number $n$ such that the block matrix
\[ \begin{pmatrix}
\sigma^s D \\
\sigma^s D^2 \\
\vdots \\
\sigma^s D^n
\end{pmatrix}  \in \Mat_{nd\times d}(K\{\sigma\})
\]
 has rank $d$ modulo $\sigma^s$, where $s$ is the maximum $\tau$-degree of all the entries of the matrices $D,\ldots, D^n$.
\end{prop}

\begin{proof}
As a $K$-module, $\mot$ is generated by the set $\{ \tau^j\kappa_i \mid i=1,\ldots, d, j\geq 0\}$. As it is 
finitely generated as $K[t]$-module, there exist numbers $r_1,\ldots, r_d$ such that
$\{ \tau^j\kappa_i \mid i=1,\ldots, d, 0\leq j\leq r_i-1\}$ generates $\mot$ as $K[t]$-module, and hence also $\{ \tau^j\kappa_i \mid i=1,\ldots, d, 0\leq j\leq r-1\}$ for $r=\max \{r_i \mid 1\leq i\leq d\}$ generates $\mot$.
Therefore, there exists a matrix $B\in \Mat_{d\times d}(K[t]\{\tau\})$ of $\tau$-degree less than $r$ such that
\[ \svect{\tau^r\kappa}{d} = B \svect{\kappa}{d} . \]
Writing $B=\sum_{k=0}^n B_kt^k$ for $B_k\in \Mat_{d\times d}(K\{\tau\})$ ($k=0,\ldots, n$), and applying Equation \eqref{eq:t-action-on-mot}, we obtain
\[ \tau^r\one_d\cdot \svect{\kappa}{d} = (\sum_{k=0}^n B_k D^k)\cdot \svect{\kappa}{d}. \]
As $(\kappa_1,\ldots,\kappa_d)$ is a basis of the free $K\{\tau\}$-module, this implies that
$\tau^r\one_d= \sum_{k=0}^n B_k D^k$.
Multiplying by $\sigma^{r+s-1}$ from the left, we get 
\begin{eqnarray*}
 \sigma^{s-1}\one_d &=& \sum_{k=0}^n \sigma^{r-1}B_k^{(-s)}\cdot \sigma^s D^k \\
&=& \sigma^{r-1}B_0^{(-s)}\cdot \sigma^s\one_d +
\left(\sigma^{r-1}B_1^{(-s)},\, \sigma^{r-1}B_2^{(-s)},\,\ldots,\, \sigma^{r-1}B_n^{(-s)}\right)\cdot \begin{pmatrix}
\sigma^s D \\
\sigma^s D^2 \\
\vdots \\
\sigma^s D^n
\end{pmatrix}.
\end{eqnarray*}
By definition of $r$ and $s$, all the matrices $\sigma^{r-1}B_k^{(-s)}$ and $\sigma^sD^k$ lie in
$\Mat_{d\times d}(K\{\sigma\})$.
Since $\sigma^{s-1}\one_d$ has rank $d$ modulo $\sigma^s$, and $\sigma^s\one_d\equiv 0$ modulo $\sigma^s$, we obtain that
also the block matrix $\transp{\begin{pmatrix}
\sigma^s D, &
\sigma^s D^2,&
\ldots \, , &
\sigma^s D^n\end{pmatrix}}$ has rank $d$ modulo $\sigma^s$.
\end{proof}

Similarly, we get a criterion if $E$ is $t$-finite.

\begin{prop}\label{prop:necessary-criterion-dumot}
Let $E$ be given as above, and assume that $E$ is $t$-finite. Then there exists a natural number $n$ such that the block matrix
\[ \begin{pmatrix}
\sigma^s D, &
\sigma^s D^2, &
\cdots , &
\sigma^s D^n
\end{pmatrix}  \in \Mat_{d\times nd}(K\{\sigma\})
\]
has rank $d$ modulo $\sigma^s$, where $s$ is the maximum $\tau$-degree of all the entries of the matrices $D,\ldots, D^n$.
\end{prop}

\begin{proof}
The proof for the dual $t$-motive is almost identical to the previous one, but with sides swapped. We briefly sketch the main steps, pointing out similarities and differences.

As a $K$-module, $\dumot$ is generated by the set $\{ \dk_i\tau^j \mid i=1,\ldots, d, j\geq 0\}$. As it is 
finitely generated as $K[t]$-module, there exist numbers $r_1,\ldots, r_d$ such that
$\{ \dk_i\tau^j \mid i=1,\ldots, d, 0\leq j\leq r_i-1\}$ generate $\dumot$ as $K[t]$-module, and hence also $\{ \dk_i\tau^j \mid i=1,\ldots, d, 0\leq j\leq r-1\}$ for $r=\max \{r_i \mid 1\leq i\leq d\}$.
Therefore, there exists a matrix $B\in \Mat_{d\times d}(K[t]\{\tau\})$ of $\tau$-degree less than $r$ such that
\[ \zvect{\dk}{d}\cdot \tau^r\one_d = \zvect{\dk}{d}\cdot B . \]
Writing $B=\sum_{k=0}^n t^k B_k $ for $B_k\in \Mat_{d\times d}(K\{\tau\})$ ($k=0,\ldots, n$), and applying Equation \eqref{eq:t-action-on-dumot}, we obtain
\[ \zvect{\dk}{d}\cdot \tau^r\one_d  = \zvect{\dk}{d}\cdot (\sum_{k=0}^n D^k B_k  ). \]
As $(\dk_1,\ldots,\dk_d)$ is a basis of the free $K\{\tau\}$-module, this implies that
$\tau^r\one_d= \sum_{k=0}^n D^k B_k $.
Multiplying by $\sigma^s$ from the left and $\sigma^{r-1}$ from the right, we get 
\begin{eqnarray*}
 \sigma^{s-1}\one_d &=& \sum_{k=0}^n \sigma^{s}D^k\cdot  \sigma^{r-1}B_k^{(r-1)} \\
&=& \sigma^{s}\one_d\cdot \sigma^{r-1}B_0^{(r-1)}  +
\begin{pmatrix}
\sigma^s D, &
\sigma^s D^2, &
\cdots , &
\sigma^s D^n
\end{pmatrix} \cdot 
\begin{pmatrix} \sigma^{r-1}B_1^{(r-1)}\\ \sigma^{r-1}B_2^{(r-1)}\\ \vdots\\ \sigma^{r-1}B_n^{(r-1)} \end{pmatrix}.
\end{eqnarray*}
By definition of $r$ and $s$, all the matrices $\sigma^{r-1}B_k^{(r-1)}$ and $\sigma^{s}D^k$ lie in
$\Mat_{d\times d}(K\{\sigma\})$.
Since $\sigma^{s-1}\one_d$ has rank $d$ modulo $\sigma^s$, and $\sigma^{s}\one_d\equiv 0$ modulo $\sigma^s$, we obtain that
also the block matrix $\begin{pmatrix}
\sigma^{s} D, &
\sigma^{s} D^2,&
\ldots \, , &
\sigma^{s} D^n\end{pmatrix}$ has rank $d$ modulo $\sigma^s$.
\end{proof}

We finally obtain our main theorem.

\begin{thm}\label{thm:main-theorem}
For the $t$-module $(E,\phi)$ with $\phi_t$ being represented by the matrix $D$, the following are equivalent
\begin{enumerate}
\item $E$ is abelian,
\item $E$ is $t$-finite,
\item the Newton polygon of the last invariant factor $\lambda_d$ of the matrix $D$ has positive slopes only. 
\end{enumerate}
\end{thm}

\begin{proof}
Combining Prop.~\ref{prop:sufficient-criterion}, Prop.~\ref{prop:necessary-criterion-mot}, Prop.~\ref{prop:necessary-criterion-dumot}, and Theorem \ref{thm:equvialent-matrix-conditions}, 
we see that being abelian and being $t$-finite are both equivalent to the equivalent conditions given in
Theorem \ref{thm:equvialent-matrix-conditions}, in particular to the condition on the Newton polygon.
\end{proof}

\section{Purity}\label{sec:purity}

In \cite{ga:tm}, Anderson defined when an abelian $t$-module and its $t$-motive are pure, and defined the weight of such a pure $t$-motive, which we both recall here.

\begin{defn}(see \cite[pp.467\& 468]{ga:tm})\\
Let $(E,\phi)$ be an abelian $t$-module, and $\mot=\mot(E)$ its $t$-motive. The $t$-motive $\mot$ and the $t$-module $E$ are called \emph{pure}, if there exists a $K\ps{\frac{1}{t}}$-lattice $\tilde{\Lambda}$ in $K\ls{\frac{1}{t}}\otimes_{K[t]} \mot$, as well as positive integers $u,v\in \NN$ such that
\[  t^u \tilde{\Lambda} = \tau^v  \tilde{\Lambda}.\]
The weight $w(\mot)$ of the $t$-motive is defined to be
\[ w(\mot)= \frac{\dim(E)}{\rk(E)}, \]
where $\dim(E)$ is the dimension of $E$ (equal to the rank of $\mot$ as $K\{\tau\}$-module), and $\rk(E)$ is the rank of $E$ (defined as the rank of $\mot$ as $K[t]$-module).
\end{defn}

Anderson also showed  that for a pure $t$-module $E$, the ratio $\frac{u}{v}$ of the numbers above equals the weight $w(\mot)$ (cf.~\cite[Lemma 1.10.1]{ga:tm}).

The main theorem of this section is
\begin{thm}\label{thm:purity}
Let $(E,\phi)$ be an abelian $t$-module of dimension $d$, and $D\in \Mat_{d\times d}(K\{\tau\})$ the matrix representing $\phi_t$ with respect to a fixed coordinate system. Let $\mot$ be the $t$-motive of $E$.
The $t$-motive $\mot$ is pure if and only if the Newton polygon of the last invariant factor $\lambda_d$ of $D$ has exactly one edge.

In this case, the weight of $\mot$ equals the reciprocal of the slope of the edge.
\end{thm}

\begin{rem}
\begin{enumerate}
\item We don't use the weight of $E$ here, as it differs in literature. Anderson defined it to be equal to $w(\mot)$ in \cite{ga:tm}, whereas in \cite{uh-akj:pthshcff}, the notion of $w(E)$ differs from that one by the sign. The negative sign in \cite{uh-akj:pthshcff} is used in order to have compatibility with the weight of the dual $t$-motive, and the isomorphism between the category theoretical dual of the $t$-motive and the dual $t$-motive (see \cite[Theorem 5.13]{uh-akj:pthshcff}).
\item We should remark that by Theorem \ref{thm:main-theorem} and Corollary \ref{cor:main-theorem-A-modules}, some aspects in \cite{uh-akj:pthshcff} simplify as abelian $t$-modules are also $t$-finite and vice versa. In particular, the canonical homomorphism $\Xi$ given in Theorem 5.13 ibid.~is an isomorphism under the given hypothesis that $E$ is abelian.
\item By \cite[Theorem 5.29]{uh-akj:pthshcff}, for a $t$-module $E$ which is both abelian and $t$-finite (so by Theorem \ref{thm:main-theorem}, abelian or $t$-finite), its $t$-motive is pure if and only if its dual $t$-motive is pure, and their weights just differ by the sign.
\end{enumerate}
\end{rem}

From Theorem \ref{thm:purity} and the last item of the previous remark, we directly obtain
\begin{cor}\label{cor:purity-dual}
Let $(E,\phi)$ be a $t$-finite $t$-module of dimension $d$, and $D\in \Mat_{d\times d}(K\{\tau\})$ the matrix representing $\phi_t$ with respect to a fixed coordinate system. Let $\dumot$ be the dual $t$-motive of $E$.
The dual $t$-motive $\dumot$ is pure if and only if the Newton polygon $N_{\lambda_d}$ of the last invariant factor $\lambda_d$ of $D$ has exactly one edge.

In this case, for the weight of $\dumot$, we have
\[ w(\dumot)=-\frac{1}{s}, \] 
where $s$ is the slope of the edge of $N_{\lambda_d}$.
\end{cor}

The proof of Theorem \ref{thm:purity} will take up the rest of this section. We haven't dealt with the Laurent series field $K\ls{\frac{1}{t}}$ and the
vector space $K\ls{\frac{1}{t}}\otimes_{K[t]} \mot$, yet. For making use of the previous sections, the main preparation for the proof is to make a connection between $K\ls{\frac{1}{t}}\otimes_{K[t]} \mot$ and $K\sls{\sigma}\otimes_{K\{\tau\}} \mot$. This culminates in Prop.~\ref{prop:iso-of-completions} showing that these two modules are isomorphic for an abelian $t$-module $E$. 

\medskip

We use the notation from Section \ref{sec:t-modules}, and assume throughout that the $t$-module $E$ is abelian. In particular,
 $\{\kappa_1,\ldots, \kappa_d\}$ is a fixed $K\{\tau\}$-basis of the $t$-motive $\mot$, and $D\in \Mat_{d\times d}(K\{\tau\})$ is the matrix representing multiplication by $t$, i.e.~
\[ t\cdot \svect{\kappa}{d} = D\cdot \svect{\kappa}{d}.\]
Further we let  $\mothat=K\sls{\sigma}\otimes_{K\{\tau\}} \mot$ which is a $K\sls{\sigma}$-vector space with basis given by $\{\kappa_1,\ldots, \kappa_d\}$, and we denote by $\stlatt\subseteq \mothat$ the $K\sps{\sigma}$-lattice generated by $\{\kappa_1,\ldots, \kappa_d\}$.
Be aware that we have a direct sum decomposition of $K$-vector spaces
\begin{equation}\label{eq:mot+lambda}
\mothat= \mot \oplus \sigma\stlatt.
\end{equation}
%This is easily seen by writing every element in $\mothat$ as a $K\sls{\sigma}$-linear combination of the $\kappa_k$.

When speaking of convergence in $\mothat$, we mean convergence with respect to the norm given by
\[  \norm{\sum_{i=1}^d x_i\kappa_i} :=\max\{ \rho^{v(x_i)} \mid i=1,\ldots,d \} \]
for some $0<\rho<1$. Here $x_1,\ldots, x_d\in K\sls{\sigma}$, and $v(x)=\ord_\sigma(x)$ denotes the valuation given in Example \ref{exmp:skew-Laurent-series-ring}. As in the case of commutative coefficient rings, $\mothat$ is complete with respect to this norm, since $K\sls{\sigma}$ is complete.
A coordinate-free description of the convergence is given by the following:\footnote{The equivalence of this description with the definition is easily verified.}\\
A sequence $(m_n)_{n\geq 0}$ of elements in $\mothat$ is converging if and only if for all $k\geq 0$, there is $n_k\geq 0$ such that
\[ m_{n+1}-m_n\in \sigma^k\stlatt \quad \forall n\geq n_k.\] 

\begin{lem}\label{lem:ps-z-action}
\begin{enumerate}
\item \label{lem:ps-z-action:item:1} There is some $n_0\in \NN$ such that 
\[\forall k\geq 1, \forall n\geq kn_0:\quad t^{-n}\stlatt \subseteq \sigma^k\stlatt.\]
\item \label{lem:ps-z-action:item:2} The $K\sps{\sigma}$-submodule $\stlatt'=\sum_{i\geq 0} t^{-i}\stlatt$ of $\mothat$ is a $K\sps{\sigma}$-lattice satisfying $t^{-1}\stlatt'\subseteq \stlatt'$.
\item \label{lem:ps-z-action:item:3} The $t$-action on $\mot$ extends to an action of $K\ls{\frac{1}{t}}$ on $\mothat$.\footnote{Of course this action does not commute with the $K\sls{\sigma}$-action, as $K$ does not commute with $\sigma$, but Laurent series in $\Fq\ls{\frac{1}{t}}$ do.}
\item \label{lem:ps-z-action:item:4} Let $m_1,\ldots, m_k\in \mothat$, and let $(f_j^{(n)})_{n\geq 0}$ for $j=1,\ldots k$ be sequences in $K\ls{\frac{1}{t}}$ that converge to some element $f_j\in K\ls{\frac{1}{t}}$ with respect to the $t^{-1}$-adic topology. Then the sequence $\left(\sum_{j=1}^k f_j^{(n)}m_j\right)_{n\geq 0}$ of elements in $\mothat$ converges to $\sum_{j=1}^k f_jm_j$ (with respect to $\norm{.}$).
\end{enumerate}
\end{lem}

\begin{proof}
\begin{enumerate}
\item Since $E$ is abelian, all slopes of the invariant factors $\lambda_1,\ldots, \lambda_d$ are positive by Thm.~\ref{thm:main-theorem}. Hence, we can apply Cor.~\ref{cor:positive-slopes} to the standard lattice $\stlatt$, and obtain $n_0\in \NN$ such that $t^n\stlatt\supseteq \sigma^{-1}\stlatt$ for $n\geq n_0$, or equivalently, $t^{-n}\stlatt\subseteq \sigma\stlatt$.

Further inductively, for all $k> 1$, for all $n\geq kn_0$:
\[ t^{-n}\stlatt = t^{-n+n_0}t^{-n_0}\stlatt \subseteq t^{-n+n_0}\sigma\stlatt \subseteq \sigma^{k-1}\sigma\stlatt =\sigma^k\stlatt.\]
\item It only has to be shown that $\stlatt'$ is a finitely generated $K\sps{\sigma}$-submodule. By the first part, however,
$t^{-n}\stlatt \subseteq \sigma^k\stlatt\subseteq\stlatt$ for all $n\geq n_0$, and hence $\stlatt'$ equals the finite sum $\sum_{i=0}^{n_0-1} t^{-i}\stlatt$, and hence is finitely generated.
\item By the first part and the description of convergence, for any $f=\sum_{i=i_0}^\infty \alpha_i t^{-i}\in K\ls{\frac{1}{t}}$, and $m\in \mothat$, the sequence
$(\sum_{i=i_0}^n \alpha_i t^{-i}m)_{n\geq 0}$ is converging, and hence
\[ f\cdot m := \lim_{n\to \infty} \sum_{i=i_0}^n \alpha_i t^{-i}m \]
is well-defined. Verifying that this indeed describes a $K\ls{\frac{1}{t}}$-action is a standard computation.

\item As $\stlatt$ is a lattice in $\mothat$, there exists $l\in\ZZ$ such that $m_1,\ldots, m_k\in \sigma^l\stlatt$. For any $n\geq 0$, let 
\[ q_n:=\min \left\{ \ord_{1/t}( f_j^{(n)}-f_j ) \bigmid j=1,\ldots, k\right\} \]
respectively $q_n=n$, if $f_j^{(n)}=f_j$ for all $j$. Then by definition of the $t^{-1}$-adic convergence, the sequence $(q_n)_{n\geq 0}$ tends to infinity. By the previous parts, for all  $n$ where $q_n>0$, we have
\[ \sum_{j=1}^k f_j^{(n)}m_j - \sum_{j=1}^k f_jm_j = \sum_{j=1}^k (f_j^{(n)}-f_j) m_j \in t^{-q_n}\sigma^l\stlatt
\subseteq \sigma^{\lfloor \frac{q_n}{n_0}\rfloor+l}\stlatt. \]
Therefore, 
\[ \lim_{n\to \infty} \norm{\sum_{j=1}^k f_j^{(n)}m_j - \sum_{j=1}^k f_jm_j}\leq \lim_{n\to \infty} \rho^{\lfloor \frac{q_n}{n_0}\rfloor+l}=0,\]
i.e., $\sum_{j=1}^k f_j^{(n)}m_j$ converges to $\sum_{j=1}^k f_jm_j$.
\end{enumerate}
\end{proof}

%Recall that $\stlatt\subseteq \mothat$ denotes the $K\sps{\sigma}$-lattice generated by $\{\kappa_1,\ldots, \kappa_d\}$.
From now on, we also fix a $K[t]$-basis $\{b_1,\ldots,b_r\}$ of the $t$-motive $\mot$. This is also a $K\ls{\frac{1}{t}}$-basis of $K\ls{\frac{1}{t}}\otimes_{K[t]} \mot$.

\begin{lem}\label{lem:bound-beta}
Let $\Lambda$ be a $K\sps{\sigma}$-lattice in $\mothat$.
\begin{enumerate}
\item The intersection $\mot\cap \Lambda$ is a finite dimensional $K$-vector space.
\item There exists $\beta_\Lambda\in \NN$ with the following property:
For all $g_1,\ldots, g_r\in K[t]$ such that $\sum_{j=1}^r g_jb_j\in \mot\cap \Lambda$, one has $\deg_t(g_j)\leq \beta_\Lambda$ for $j=1,\ldots, r$.
\end{enumerate}
\end{lem}

\begin{proof}
\begin{enumerate}
\item As $\Lambda$ is a $K\sps{\sigma}$-lattice in $\mothat$, there is some $l\geq 0$ such that
$\Lambda \subseteq \sigma^{-l}\stlatt$. Furthermore by definition, the intersection $\mot\cap \sigma^{-l}\stlatt$ consists of all elements of the form  $\sum_{k=1}^d \sum_{i=-l}^0 \alpha_{ki}\sigma^{i}\kappa_k$, and hence is a finite dimensional $K$-vector space. Therefore, also $\mot\cap \Lambda$ is finite dimensional.
\item By definition, $\{b_1,\ldots, b_r\}$ is a $K[t]$-basis of $\mot$. Hence, each element in $\mot\cap \Lambda$ can by uniquely written as $\sum_{j=1}^r g_jb_j$ or some $g_1,\ldots, g_r\in K[t]$. As the intersection is a finite dimensional $K$-vector space, there is an upper bound on the degrees of the $g_j$ that appear in these representations.
\end{enumerate}
\end{proof}

\begin{lem}\label{lem:approximation}
Let $m\in \mothat$ be arbitrary.
\begin{enumerate}
\item \label{item:1:lem:approximation} For all $n\geq 0$, there exist unique $f_1^{(n)},\ldots, f_r^{(n)}\in K\ls{\frac{1}{t}}$ such that
\begin{enumerate}
\item $t^nf_j^{(n)}\in K[t]$ for $j=1,\ldots, r$, and 
\item $t^n\cdot \left( \sum_{j=1}^r f_j^{(n)}b_j -m\right) \in \sigma\stlatt$.
\end{enumerate}
\item Let $f_j^{(n)}\in K\ls{\frac{1}{t}}$ be the elements determined in part \eqref{item:1:lem:approximation}. Then
\begin{enumerate}
\item For all $j=1,\ldots, r$, the sequence $(f_j^{(n)})_{n\geq 0}$ converges to some element $f_j\in K\ls{\frac{1}{t}}$ with respect to the $t^{-1}$-adic topology, and
%\item $\ord_{1/t}\left( f_j^{(n+1)}- f_j^{(n)}\right) \geq n+1-\beta$ for all $n\geq 0$, as well as
\item one has $\lim_{n\to \infty} \sum_{j=1}^r f_j^{(n)}b_j=\sum_{j=1}^r f_jb_j=m $ in $\mothat$.
\end{enumerate}
\end{enumerate}
\end{lem}

\begin{proof}
\begin{enumerate}
\item Let $n\geq 0$. Using the decomposition \eqref{eq:mot+lambda}, we write $t^n m = m_\Lambda +m_\mot$ with uniquely determined
$m_\Lambda\in\sigma\stlatt$ and $m_\mot\in \mot$. As $\mot$ is a free $K[t]$-module with basis $\{b_1,\ldots, b_r\}$, there are unique $g_1,\ldots, g_r\in K[t]$ such that $m_\mot = \sum_{j=1}^r g_jb_j$. Hence, the elements $f_j^{(n)}:=t^{-n}g_j$ for $j=1,\ldots, r$ satisfy the two given properties.\\
On the other hand, these two properties imply that $t^n m= \left( \sum_{j=1}^r t^n f_j^{(n)}b_j -t^n m\right) +  \left( \sum_{j=1}^r t^n f_j^{(n)}b_j\right)$ is the unique decomposition of $t^n m$ into an element in $\sigma\stlatt$ and another in $\mot$. Hence, the constructed elements  $f_j^{(n)}$ are unique.
\item By the first condition on the $f_j^{(n)}$ and $f_j^{(n+1)}$, we have 
\[t^{n+1}\left( f_j^{(n+1)}- f_j^{(n)}\right)\in K[t],\] and hence 
\[ t^{n+1}\sum_{j=1}^r  \left( f_j^{(n+1)}- f_j^{(n)}\right)b_j\in \mot. \]
On the other hand, 
\begin{eqnarray*} 
&& t^{n+1}\sum_{j=1}^r  \left( f_j^{(n+1)} - f_j^{(n)}\right)b_j\\
 && \quad = t^{n+1}\cdot \left( \sum_{j=1}^r f_j^{(n+1)}b_j -m\right)
- t\cdot t^n \cdot \left( \sum_{j=1}^r f_j^{(n)}b_j -m\right) \in \sigma\stlatt + t\sigma\stlatt, 
\end{eqnarray*}
by the second property. So let $\Lambda=\sigma\stlatt+t\sigma\stlatt$, and let $\beta_\Lambda\in \NN$ be the minimal number satisfying the given property in Lemma \ref{lem:bound-beta}. 
We therefore have
\[ \deg_t\left( t^{n+1}\cdot ( f_j^{(n+1)}- f_j^{(n)}) \right)\leq \beta_\Lambda,\]
and hence
\[ \ord_{1/t} ( f_j^{(n+1)}- f_j^{(n)})\geq -\beta_\Lambda - \ord_{1/t} (t^{n+1})=-\beta_\Lambda+ n+1.\]
This shows that the sequence $(f_j^{(n)})_{n\geq 0}$ converges $t^{-1}$-adically to some element $f_j\in K\ls{\frac{1}{t}}$.

By Lemma \ref{lem:ps-z-action}\eqref{lem:ps-z-action:item:4}, this implies that $\sum_{j=1}^r   f_j^{(n)} b_j$ converges in $\mothat$ to the element
$\sum_{j=1}^r f_jb_j$.

Furthermore, by the given property, for all $n\geq 0$
\[ t^{n}\left( \sum_{j=1}^r   f_j^{(n)} b_j - m\right) \in \sigma\stlatt,\]
and hence, 
\[ \sum_{j=1}^r   f_j^{(n)} b_j - m \in t^{-n}\sigma\stlatt\subseteq \sigma^{\lfloor \frac{n}{n_0}\rfloor +1}\stlatt,\]
where $n_0$ as in Lemma \ref{lem:ps-z-action}.
Hence, the sequence $(\sum_{j=1}^r f_j^{(n)} b_j)_{n\geq 0}$ converges to $m$.
\end{enumerate}
\end{proof}

\begin{prop}\label{prop:iso-of-completions} \
\begin{enumerate} 
\item We have a natural isomorphism of $K\ls{\frac{1}{t}}\{\tau\}$-modules $K\ls{\frac{1}{t}}\otimes_{K[t]} \mot \cong \mothat$ given by
\[ \sum_{j=1}^r f_j\otimes b_j \mapsto \sum_{j=1}^r f_j\cdot b_j \]
for all $f_1,\ldots, f_r\in  K\ls{\frac{1}{t}}$.
\item \label{prop:iso-of-completions:item:2} If $\Lambda$ is a $K\sps{\sigma}$-lattice in $\mothat$ such that $t^{-1}\Lambda\subseteq \Lambda$, then $\Lambda$ is a $K\ps{\frac{1}{t}}$-lattice in $\mothat$.
\item \label{prop:iso-of-completions:item:3} If $\Lambda$ is a $K\ps{\frac{1}{t}}$-lattice in $\mothat$ such that $\sigma\Lambda\subseteq \Lambda$, then $\Lambda$ is a $K\sps{\sigma}$-lattice in $\mothat$.
\end{enumerate}
\end{prop}

\begin{proof}
\begin{enumerate}
\item By Lemma \ref{lem:ps-z-action}, we have a $K\ls{\frac{1}{t}}$-action on $\mothat$, and hence the map is well defined.
As the left hand side is the $K\ls{\frac{1}{t}}$-vector space with basis $\{b_1,\ldots, b_r\}$, it suffices to show that every element in $\mothat$ can uniquely be written as $\sum_{j=1}^r f_j b_j$ with $f_1,\ldots, f_r\in  K\ls{\frac{1}{t}}$. This, however, was just proven in Lemma \ref{lem:approximation}.
%\item By Lemma \ref{lem:ps-z-action}(3), the standard lattice $\stlatt$ is a $K\ps{\frac{1}{t}}$-submodule of $\mothat$. We will verify the conditions given in Lemma \ref{lem:O-lattice-description} for $\cO=K\ps{\frac{1}{t}}$, and $\cL=K\ls{\frac{1}{t}}$.
%Since $\stlatt$ is a $K\sps{\sigma}$-lattice, there is some $l\geq 0$ such that the $K\ls{\frac{1}{t}}$-basis $b_1,\ldots, b_r$  of $\mothat$ is in $\sigma^{-l}\stlatt$. Then
%\[ t^{-l}b_j \in  t^{-l}\sigma^{-l}\stlatt \subseteq \sigma^l\sigma^{-l}\stlatt=\stlatt.\]
%Hence, $\stlatt$ contains a $K\ls{\frac{1}{t}}$-basis of $\mothat$.
%
%Let $m\in \mothat$, $m\ne 0$ be an arbitrary element. Assume that $x\cdot m\in \stlatt$ for all $x\in K\ls{\frac{1}{t}}$. Then in particular, $t^n m\in \stlatt$ for all $n\geq 0$, and hence $t^n\sigma m\in \sigma\stlatt$.
%However, this implies that the unique sequences $(f_j^{(n)})$ attached to $\sigma m$ by Lemma \ref{lem:approximation} are all identically zero. Therefore, $\sigma m=0$, and $m=0$ contrary to the assumption.
\item By assumption the $K\sps{\sigma}$-lattice $\Lambda$ is stable under the $t^{-1}$-action, and hence it is also a $K\ps{\frac{1}{t}}$-module. Since $\Lambda$ is a $K\sps{\sigma}$-lattice, there is some $l\geq 0$ such that the $K\ls{\frac{1}{t}}$-basis $b_1,\ldots, b_r$  of $\mothat$ is in $\sigma^{-l}\Lambda$. By Cor.~\ref{cor:positive-slopes}, there is $n_0\geq 0$ such that
$t^{n_0}\Lambda\supseteq \sigma^{-1}\Lambda$, and hence $\sigma^{-l}\Lambda\subseteq t^{ln_0}\Lambda$.
Therefore, the $K\ls{\frac{1}{t}}$-basis $(t^{-ln_0}b_1,\ldots, t^{-ln_0}b_r)$ of $\mothat$ lies inside $\Lambda$.

On the other hand, since $\Lambda$ is a $K\sps{\sigma}$-lattice, for arbitrary $m\in \mothat$, there is $l>0$ such that $\sigma^{-l}m\not\in\Lambda$, i.e.~$m\not\in\sigma^l\Lambda$.
Since $t^{-ln_0}\Lambda\subseteq \sigma^{l}\Lambda$, this implies $m\not\in t^{-ln_0}\Lambda$. Hence, $t^{ln_0}m\not\in \Lambda$.
By Lemma \ref{lem:O-lattice-description}, $\Lambda$ is a $K\ps{\frac{1}{t}}$-lattice.
\item By assumption the $K\ps{\frac{1}{t}}$-lattice $\Lambda$ is stable under the $\sigma$-action, and hence it is also a $K\sps{\sigma}$-module. Let $\stlatt'=\sum_{i\geq 0} t^{-i}\stlatt$ be the $K\sps{\sigma}$-lattice defined in Lemma \ref{lem:ps-z-action}\eqref{lem:ps-z-action:item:2}. By the previous part, it is also a $K\ps{\frac{1}{t}}$-lattice.

As $\Lambda$ and $\stlatt'$ are both $K\ps{\frac{1}{t}}$-lattices, there are $l_1,l_2\in\ZZ$ such that 
\[ t^{l_1}\stlatt' \subseteq \Lambda \subseteq t^{l_2}\stlatt'. \]
As $\stlatt'$ is a $K\sps{\sigma}$-lattice, so are $t^{l_1}\stlatt'$ and $t^{l_2}\stlatt'$, and therefore by Remark \ref{rem:lattices}, $\Lambda$ is a $K\sps{\sigma}$-lattice.
\end{enumerate}
\end{proof}

We finally prove the main theorem of this section.

\begin{proof}[Proof of Theorem \ref{thm:purity}]
As $K\ls{\frac{1}{t}}\otimes_{K[t]} \mot$ is isomorphic to $\mothat$ by Proposition \ref{prop:iso-of-completions}, the aim is to show that there is a $K\ps{\frac{1}{t}}$-lattice $\Lambda$ in $\mothat$ satisfying $t^u\Lambda=\tau^v\Lambda$ for some integers $u,v\in \ZZ$, $u\ne 0$ if and only if the last invariant factor $\lambda_d$ has only one edge, and that in this case the slope $s$ of this edge equals $\frac{v}{u}$.

By Theorem \ref{thm:elementary-divisor-theorem}, $\mothat$ is isomorphic to 
$\cL[t]/\lambda_1\cL[t]\oplus \ldots \oplus\cL[t]/\lambda_d\cL[t]$
as $\cL[t]$-modules, where $\cL=K\sls{\sigma}$. 
By Thm.~\ref{thm:decomposition-in-one-slopes}, this can be decomposed further into a direct sum $\bigoplus_{i=1}^k \cL[t]/f_i\cL[t]$, for some monic polynomials $f_i\in \cL[t]$ whose Newton polygons consist of one edge. Each $f_i$ is a divisor of some $\lambda_j$ and hence a divisor of $\lambda_d$, and all similarity classes of divisors of $\lambda_d$ occur. So if the Newton polygon of $\lambda_d$ has only one edge, all the Newton polygons of the $f_i$ have the same slope. If the Newton polygon of $\lambda_d$ has more than one edge, there are $f_i$'s whose Newton polygons have different slopes.

So assume first that the Newton polygon of $\lambda_d$ has only one edge, and let $s$ be its slope. Then the Newton polygons of all $f_i$ have slope $s$. By Cor.~\ref{cor:stable-O-lattices}, each factor $\cL[t]/f_i\cL[t]$ contains a $K\sps{\sigma}$-lattice $\Lambda_i$ satisfying $\sigma^{sd_i}t^{d_i}\Lambda_i=\Lambda_i$ as well as $t^{-1}\Lambda_i\subseteq \Lambda_i$, where $d_i=\deg_t(f_i)$.

Let $d=\mathrm{lcm}(d_1,\ldots, d_k)$ be the least common multiple of all $d_i$, then for all $i$, 
\[ \sigma^{sd}t^{d}\Lambda_i=\Lambda_i,\] 
and hence the $K\sps{\sigma}$-lattice $\Lambda := \bigoplus_{i=1}^k \Lambda_i$ in $\mothat=\bigoplus_{i=1}^k \cL[t]/f_i\cL[t]$
satisfies $\sigma^{sd}t^{d}\Lambda=\Lambda$, as well as $t^{-1}\Lambda\subseteq \Lambda$.

By Proposition \ref{prop:iso-of-completions}\eqref{prop:iso-of-completions:item:2}, $\Lambda$ is also a 
$K\ps{\frac{1}{t}}$-lattice of $\mothat$, and satisfies $\sigma^{sd}t^{d}\Lambda=\Lambda$, i.e.
\[ t^{d}\Lambda=\tau^{sd}\Lambda. \]
Therefore, the $t$-motive is pure with weight $\frac{d}{sd}=\frac{1}{s}$.

\medskip

For the converse direction,
%Now, we consider the case that the Newton polygon of $\lambda_d$ has more than one edge. Again by Cor.~\ref{cor:stable-O-lattices}, each factor $\cL[t]/f_i\cL[t]$ contains a $K\sps{\sigma}$-lattice $\Lambda_i$ satisfying $\sigma^{s_id_i}t^{d_i}\Lambda_i=\Lambda_i$ as well as $t^{-1}\Lambda_i\subseteq \Lambda_i$, where $d_i=\deg_t(f_i)$ and $s_i$ is the slope of the Newton polygon of $f_i$. 
assume that the $t$-motive $\mot$ is pure, i.e.~there is a $K\ps{\frac{1}{t}}$-lattice $\Lambda$ in $\mothat$ satisfying $t^u\Lambda=\tau^v\Lambda$ for some $u,v\in \ZZ$, $u\ne 0$.
If $\Lambda$ is not stable under the action of $\sigma$ (i.e.~$\sigma\Lambda\not\subseteq \Lambda$), we can replace $\Lambda$ by the $K\ps{\frac{1}{t}}$-lattice  
$\Lambda'=\sum_{i=0}^{v-1}\sigma\Lambda$ which 
also satisfies $t^u\Lambda'=\tau^v\Lambda'$, and further $\sigma\Lambda'\subseteq \Lambda'$ (since $\sigma^v\Lambda=t^{-u}\Lambda\subseteq \Lambda$). So without loss of generality, we can assume that $\Lambda$ satisfies $\sigma\Lambda\subseteq \Lambda$.
Therefore by Prop.~\ref{prop:iso-of-completions}\eqref{prop:iso-of-completions:item:3}, $\Lambda$ is also a $K\sps{\sigma}$-lattice in $\mothat$.

For each $i=1,\ldots, k$, the intersection of $\Lambda$ with the $i$-th factor $\cL[t]/f_i\cL[t]$ is a $K\sps{\sigma}$-lattice $\Lambda_i$ in that factor satisfying $t^u\Lambda_i=\tau^v\Lambda_i$. By Theorem \ref{thm:stable-O-lattice}\eqref{thm:stable-O-lattice:item:2}, the fraction $\frac{v}{u}$ therefore equals the slope of $f_i$. Hence, all the $f_i$ have the same slope.
\end{proof}

\section{General coefficient rings} \label{sec:general-coefficients}

In this section, let $F$ be a function field over $\Fq$, $\infty$ a place of $F$, and let $A$ be the ring of regular functions outside $\infty$. Further, let $t\in F$ be a non-constant function whose pole divisor is supported only at $\infty$.\footnote{The existence of such an element $t$ is guaranteed by the Riemann-Roch theorem.} By this choice, we have $t\in A$, and $A$ is a finite extension of $\Fq[t]$.

An (Anderson) $A$-module is defined in the same way as a $t$-module, but with $\Fq[t]$ replaced by $A$. In detail,
an $A$-module $(E,\phi)$ over $K$ of characteristic $\ell:A\to K$ is an $\Fq$-vector space scheme $E$ over $K$ isomorphic to some $\Ga^d$ together with a ring homomorphism
$\phi:A\to \End_{\grp,\Fq}(E),a\mapsto \phi_a$
such that for all $a\in A$: $d\phi_a-\ell(a):\Lie(E)\to \Lie(E)$ is nilpotent.

The \emph{$A$-motive} attached to an $A$-module $(E,\phi)$ is again
\[ \mot(E):=\Hom_{\grp,\Fq}(E,\Ga) \]
and the \emph{dual $A$-motive} is
\[ \dumot(E):=\Hom_{\grp,\Fq}(\Ga,E), \]
both carrying an $(A\otimes_{\Fq} K)$-action instead of just the $K[t]$-action.

The $A$-module $(E,\phi)$ is called abelian, if its motive $\mot(E)$ is finitely generated as $A\otimes_{\Fq} K$-module, and it is called $A$-finite, if its dual motive $\dumot(E)$ is finitely generated as $A\otimes_{\Fq} K$-module.

From Theorem \ref{thm:main-theorem}, we deduce the same equivalence for $A$-modules.

\begin{cor}\label{cor:main-theorem-A-modules}
For an Anderson $A$-module $(E,\phi)$, the following are equivalent
\begin{enumerate}
\item $(E,\phi)$ is abelian,
\item $(E,\phi)$ is $A$-finite.
\end{enumerate}
\end{cor}

\begin{proof}
By restricting the coefficients from $A$ to $\Fq[t]$, we obtain a $t$-module $E$ whose $t$-motive and dual $t$-motive are $\mot(E)$ and $\dumot(E)$ considered with the restricted action.
By Thm.~\ref{thm:main-theorem}, $\mot(E)$ is finitely generated as $K[t]$-module if and only if $\dumot(E)$
is finitely generated as $K[t]$-module. As $A/\Fq[t]$ is a finite ring extension, and hence $A\otimes_{\Fq} K/K[t]$ is a finite ring extension, a module $M$ over $A\otimes_{\Fq} K$ is finitely generated if and only if $M$ is finitely generated as $K[t]$-module. Hence, $\mot(E)$ is a finitely generated $A\otimes_{\Fq} K$-module if and only if $\dumot$ is a finitely generated $A\otimes_{\Fq} K$-module.
\end{proof}

\section{Examples} \label{sec:examples}

We demonstrate our main theorems by a few examples of $t$-modules.

\begin{exmp}
Let $\psi$ be a Drinfeld module over $K$ of rank $r$ and characteristic $\ell:\Fq[t]\to K$, i.e.~
\[ \psi_t=\ell(t)+a_1\tau+\ldots + a_r\tau^r \]
with $a_1,\ldots, a_r\in K$, $a_r\ne 0$.

The matrix $t\one_1-\psi_t$ is a $1\times 1$-matrix, and therefore already in diagonal form, and the Newton polygon of $\lambda_1=t-\psi_t$ has vertices at $(0,-r)$ and $(1,0)$.

Hence, the Newton polygon consists of one edge which has positive slope $r$, and the main theorems state the well-known result that the Drinfeld module $\psi$ is a pure abelian and $t$-finite $t$-module of weight $\frac{1}{r}$.
\end{exmp}

\begin{exmp}
Extensions of Drinfeld modules by $\GG_a$:\\
Let $\psi$ be a Drinfeld module over $K$ of rank $r$ and characteristic $\ell:\Fq[t]\to K$, and let $\delta:\Fq[t]\to \tau K\{\tau\},a\mapsto \delta_a$ be a $\ell$-$\psi$-bi-derivation, i.e.~an $\Fq$-linear map satisfying
\[ \delta_{ab}=\ell(a)\cdot \delta_b+\delta_a\cdot \psi_b \]
for all $a,b\in \Fq[t]$.
We consider the $t$-module $(E,\phi)$ of dimension $2$ given by

\[ \phi_t = \begin{pmatrix} \psi_t & 0 \\ \delta_t & \theta \end{pmatrix},\]
where $\theta=\ell(t)$.
%\[ \phi_t = \begin{pmatrix}
%\psi_t & 0 & \cdots & \cdots & 0\\ 
%\delta_1 & \theta & \ddots &&\vdots \\
%\vdots & 0 & \ddots & & \vdots \\
%\delta_s & \cdots & \cdots & & \theta
%\end{pmatrix}. \]

For $t-\phi_t$, one computes the invariant factors $\lambda_1=1$, and $\lambda_2=(t-\psi_t)\delta_t^{-1}(t-\theta)$.
The Newton polygon of $\lambda_2$ has two edges. The second edge corresponds to the factor $(t-\psi_t)$ and has slope $r$, the other corresponds to the factor $(t-\theta)$ and has slope $0$.\\
As one slope is non-positive, the $t$-module $\phi$ is not abelian and not $t$-finite.
%
%This $t$-module is not abelian
%If $\kappa_1\in \mot(\phi)$ is the projection to the first coordinate,
%$m_1=\kappa_1, m_2=\tau\kappa_1,\ldots, m_r=\tau^{r-1}\kappa_1$ is a $\TT$-basis of $\mot(\phi)\otimes_{K[t]} \TT$.
%Indeed $\mot(\psi)\otimes_{K[t]} \TT\isom \mot(\phi)\otimes_{K[t]} \TT$, and the isomorphism is induced by the projection $\phi\to \psi$ of $t$-modules.
\end{exmp}

\begin{exmp}
The $d$-th Carlitz tensor power $C^{\otimes d}$ is given by
\[ C_t =\begin{pmatrix} 
\theta & 1 & 0 & \cdots & 0\\ 
0 & \ddots & \ddots & \ddots & \vdots\\
\vdots & \ddots &  \ddots & \ddots & 0 \\
0 & & \ddots & \ddots & 1 \\
\tau & 0 & \cdots & 0 & \theta \end{pmatrix} \in \Mat_{d\times d}(K\{\tau\})
.\]
Diagonalizing $t\one_d-C_t$ in $\Mat_{d\times d}(K\sls{\sigma}[t])$, we obtain the invariant factors $1,\ldots, 1$ and $(t-\theta)^d-\sigma^{-1}$. The Newton polygon of the last polynomial has exactly one edge starting in $(0,-1)$ and ending in $(d,0)$, hence has positive slope $\frac{1}{d}$.\\
Hence, our main theorems verify that $C^{\otimes d}$ is abelian and $t$-finite, and that its $t$-motive is pure of weight $d$.
\end{exmp}

\begin{exmp}\label{exmp:new-t-module-revisited}
We consider again Example \ref{exmp:new-t-module}, i.e.~the $t$-module $(E,\phi)$ over the rational function field $K=\Fq(\theta)$ with
\[ \phi_t=
%\begin{pmatrix} \theta & 0 \\ 1 & \theta \end{pmatrix}+
%\begin{pmatrix} 0 & 0 \\ 1 &0\end{pmatrix}\cdot \tau+
%\begin{pmatrix} 1 & 0 \\ 0 &1\end{pmatrix}\cdot \tau^2+
%\begin{pmatrix} 0 & 1 \\ 0 &0\end{pmatrix}\cdot \tau^3 =
\begin{pmatrix} \theta+\tau^2 & \tau^3 \\ 1+\tau & \theta+\tau^2 \end{pmatrix}
. \]
By diagonalizing $t\one_2-\phi_t\in \Mat_{2\times 2}(K\sls{\sigma}[t])$, we obtain the invariant factors $\lambda_1=1$ and 
\[\lambda_2=t^2-\left(2\sigma^{-2}+\theta^{1/q^3}+\theta\right)\cdot t+
\left(-\sigma^{-3}+(\theta+\theta^{q^2})\sigma^{-2}+\theta^{q^3+1}\right).\]
The Newton polygon of $\lambda_2$ is depicted in Figure \ref{fig:3}. 
It has only edges with positive slopes, hence this $t$-module $E$ is abelian and $t$-finite.

\noindent
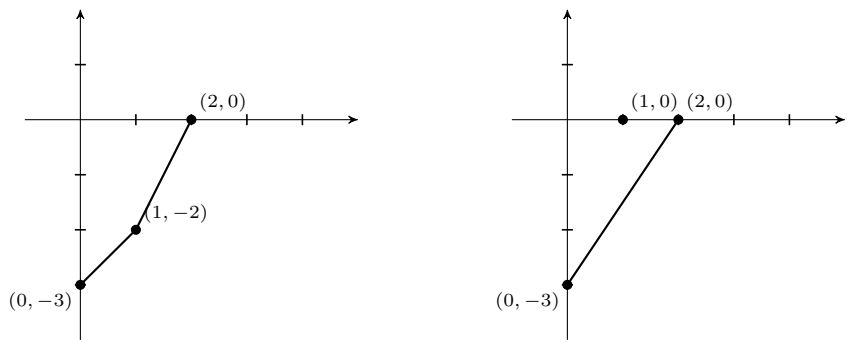
\begin{figure}[ht] 
\begin{tabular}{cp{1cm}c}
\begin{tikzpicture}[
    scale=0.73,
    axis/.style={thin, ->, >=stealth'},
    every node/.style={color=black},
    ]
	%\scriptsize 
    \tiny

	% Axis 
    \draw[axis] (-1,0)  -- (5,0) ;
    \draw[axis] (0,-4) -- (0,2) ;
    \draw[semithick] (1,-0.1) -- (1,0.1); 
    \draw[semithick] (2,-0.1) -- (2,0.1); 
    \draw[semithick] (3,-0.1) -- (3,0.1); 
    \draw[semithick] (4,-0.1) -- (4,0.1); 

    \draw[semithick] (-0.1,-1) -- (0.1,-1); 
    \draw[semithick] (-0.1,-2) -- (0.1,-2); 
    \draw[semithick] (-0.1,-3) -- (0.1,-3); 
    \draw[semithick] (-0.1,1) -- (0.1,1); 

	% Points
    \draw [fill] (0,-3) circle [radius=.081] node [below left] (0,-3) {$(0,-3)$};
    \draw [fill] (1,-2) circle [radius=.081] node [above right] (1,-2) {$(1,-2)$};
    \draw [fill] (2,0) circle [radius=.081] node [above right] (2,0) {$(2,0)$};

    % Lines
    \draw[thick] (0,-3) -- (1,-2) -- (2,0); 
\end{tikzpicture} 
& &
\begin{tikzpicture}[
    scale=0.73,
    axis/.style={thin, ->, >=stealth'},
    every node/.style={color=black},
    ]
	%\scriptsize 
    \tiny

	% Axis 
    \draw[axis] (-1,0)  -- (5,0) ;
    \draw[axis] (0,-4) -- (0,2) ;
    \draw[semithick] (1,-0.1) -- (1,0.1); 
    \draw[semithick] (2,-0.1) -- (2,0.1); 
    \draw[semithick] (3,-0.1) -- (3,0.1); 
    \draw[semithick] (4,-0.1) -- (4,0.1); 

    \draw[semithick] (-0.1,-1) -- (0.1,-1); 
    \draw[semithick] (-0.1,-2) -- (0.1,-2); 
    \draw[semithick] (-0.1,-3) -- (0.1,-3); 
    \draw[semithick] (-0.1,1) -- (0.1,1); 

	% Points
    \draw [fill] (0,-3) circle [radius=.081] node [below left] (0,-3) {$(0,-3)$};
    \draw [fill] (1,0) circle [radius=.081] node [above right] (1,0) {$(1,0)$};
    \draw [fill] (2,0) circle [radius=.081] node [above right] (2,0) {$(2,0)$};

    % Lines
    \draw[thick] (0,-3) -- (2,0); 
\end{tikzpicture} 
\end{tabular}
\caption{Newton polygon for $\lambda_2$ in characteristic $\ne 2$ and in characteristic $2$.}\label{fig:3}
\end{figure}

If the characteristic is different from $2$, $E$ is not pure. On the other hand, if $\Fq$ has characteristic $2$, $E$ is pure and the weight of its $t$-motive is $\frac{2}{3}$.
%As for further properties: It seems that this $t$-motive is not pure and not even mixed.

%A $K[t]$-basis for the $t$-motive can be computed to be $(m_1,m_2,m_3):=( \kappa_1, \kappa_2, \tau\kappa_2)$ with $\tau$-action given by
%\[ \tau \svect{m}{3} = \begin{pmatrix}
%\theta-t & 0 & t-\theta^q \\ 0 & 0 & 1\\ t-\theta-1 & t-\theta & \theta^q-t\end{pmatrix}\cdot \svect{m}{3}. \]
\end{exmp}

%$\ldots$ add more examples ?? $\ldots$

%\enlargethispage{10mm}

%---------------------------------------------------------------------
% Ende des eigentlichen Artikels
%---------------------------------------------------------------------
\bibliographystyle{alpha}
%\bibliography{../reference}
\def\cprime{$'$}

%\vspace*{.25cm}

\parindent0cm

\end{document}